\documentclass[11pt]{amsart}
\usepackage[english]{babel}
\usepackage{amssymb,amsmath,amsthm}
\usepackage{yhmath}
\usepackage{color}

\newtheorem{theorem}{Theorem}[section]

\newtheorem{lem}[theorem]{Lemma}
\newtheorem{pro}[theorem]{Proposition}
\newtheorem{prop}[theorem]{Proposition}
\newtheorem{proposition}[theorem]{Proposition}
\newtheorem{coro}[theorem]{Corollary}

\theoremstyle{definition}

\newtheorem{example}[theorem]{Example}

\newtheorem{rem}[theorem]{Remark}

\theoremstyle{plain}
\newtheorem*{namedthm}{\namedthmname}
\newcounter{namedthm}



 \renewenvironment{proof}{{\bfseries Proof.}}{\qed}

\newcommand{\R}{\mathbb{R}}
\newcommand{\C}{\mathbb{C}}
\newcommand{\N}{\mathbb{N}}

\newcommand{\fact}[1]{#1\mathpunct{}!}

\newcommand{\f}{\varphi}

\numberwithin{equation}{section}

\usepackage{hyperref}
\hypersetup{
    bookmarks=true,         
    unicode=false,         
    pdftoolbar=true,       
    pdfmenubar=true,       
    pdffitwindow=false,    
    pdfstartview={FitH},   
    colorlinks=true,       
   linkcolor=black,         
    citecolor=black,        
    filecolor=black,      
    urlcolor=black}          

\title[ Eigenvalue problem]{The  Eigenvalue Problem for the \\ complex Monge-Amp\`ere operator}

\author{Papa Badiane}
\address{Laboratoire de Mathématiques et Applications;  Université Assane Seck de Ziguinchor BP 523.}
\email{p.badiane4963@zig.univ.sn}

\author{Ahmed Zeriahi}

\address{Institut de Mathématiques de Toulouse; UMR 5219, Université de Toulouse; CNRS, UPS, 118 route de Narbonne, F-31062 Toulouse Cedex 9, France}

\email{ahmed.zeriahi@math.univ-toulouse.fr}



\begin{document}

\date{\today}
\maketitle  

\setcounter{tocdepth}{1}

\begin{abstract}
 We prove the existence of the first eigenvalue  and an associated eigenfunction with Dirichlet condition for the complex Monge-Amp\`ere operator  on a bounded strongly pseudoconvex domain in $\C^n$. We show that the eigenfunction is  plurisubharmonic,  smooth with bounded Laplacian in $\Omega$  and boundary values $0$. Moreover it is unique up to a positive multiplicative constant.
   
  To this end, we follow the strategy used by P.L. Lions in the real case. However, we have to prove a new  theorem on the existence of solutions for some special complex degenerate Monge-Ampère equations. This requires establishing new a priori estimates of the gradient and Laplacian of such solutions using methods and results of L. Caffarelli, J.J. Kohn, L. Nirenberg and J. Spruck  \cite{CKNS85} and B. Guan \cite{Guan98}.
  
   Finally we provide a Pluripotential variational approach to the problem and using our new existence theorem, we prove a Rayleigh quotient type formula for the first eigenvalue of the complex Monge-Ampère operator.
 \end{abstract}

 \section{introduction}

Let $\Omega\Subset \C^n$ be a bounded strongly pseudoconvex domain of $\C^n$ with smooth boundary $\partial \Omega$ and $0<f \in C^{\infty}(\bar \Omega)$.

 Our goal is to solve the eigenvalue problem with Dirichlet boundary condition for a twisted complex Monge-Ampère operator. 
 
 The problem consists in finding a couple $(\lambda,u)$, with $\lambda > 0$ and $u \in PSH (\Omega) \cap C^0 (\bar \Omega) \cap C^2(\Omega)$ satisfying the following conditions:

\begin{equation}\label{eq1}
\left\lbrace
\begin{array}{lcr}
(dd^cu)^n=(-\lambda u)^n f^n\omega^n & \textnormal{in} & \Omega\\
u=0& \textnormal{ on} & \partial \Omega\\
\Vert u\Vert_{C^0(\bar \Omega)} = 1,& &
\end{array}
\right.
\end{equation}
where  $\omega:= dd^c\vert z \vert^2$ is the standard Kähler form on $\C^n$. 
Here we use the standard differential operators $d = \partial + \partial$ and $d^c := (i \slash 2) (\bar \partial - \partial)$ so that $dd^c = i \partial \bar \partial$.

P. L. Lions solved the existence problem of the first eigenvalue  with  Dirichlet boundary condition for the real Monge-Ampère operator on a smooth bounded strongly convex domain in $\R^N$ (see \cite{Lions86}). He claimed that using the results of L. Caffarelli, J.J. Kohn, L. Nirenberg and J. Spruck  \cite {CKNS85}, all his results extend without changes to the case of the complex Monge-Ampère operator. 

 We were not able to directly apply the results of \cite{CKNS85} as claimed by Lions to solve the problem in the complex case.
Indeed we have to deal with complex Monge-Ampère equations with  a right hand side of the type $\psi_\varepsilon (z,u) := (\varepsilon-\lambda u)^n f(z)^n$ with $\varepsilon > 0$. Since these functions are decreasing  in $u$ and $\psi_\varepsilon (z,0) \equiv \varepsilon f(z)^n$ is completely degenerate at the boundary as $\varepsilon \to 0^+$, the existence results  of \cite{CKNS85}  and \cite{Guan98} do not apply directly in this case. 

 Instead, we will establish  a  new existence theorem  for degenerate complex Monge-Amp\`ere equations of this type  (see Theorem \ref{thm:new})  based on the fundamental theorem of \cite{CKNS85}  and  new  Laplacian a priori estimates for the inverse Monge-Amp\`ere operator.

\smallskip

Let us state our first main result.
\begin{theorem}\label{thm1}
	Let $\Omega \Subset \C^n$ be a bounded strongly pseudoconvex domain with smooth boundary and $0 < f \in C^{\infty}(\bar \Omega)$ be a smooth positive  function in $\bar \Omega$. 
	
	Then there exists a real number $\lambda_1 = \lambda_1(\Omega,f)>0$ and a function $u_1 \in PSH(\Omega)\cap C^{\infty}(\Omega)  \cap C^{1,\bar 1} (\bar \Omega)$ with $\Vert u_1 \Vert_{C^0(\bar \Omega)} = 1$ such that the pair $(\lambda_1, u_1)$ is the unique  solution to the eigenvalue problem \eqref{eq1}. 
\end{theorem}
The number $\lambda_1$ will be called the first eigenvalue of the complex Monge-Ampère operator with respect to the volume form  $ d\nu_f := f^n \omega^n$ and $u_1$ will be called its  associated normalized eigenfunction.

\smallskip

Here $C^{1,\bar 1}(\bar \Omega)$ denotes the space of all function $u \in C^{1}(\bar \Omega)$ such that its Laplacian $\Delta u$ in the sense of distributions is a bounded function $\Delta u \in L^{\infty} (\Omega)$  i.e. 
$$
\Vert u \Vert_{C^{1,\bar 1}( \bar \Omega)}:= \Vert u \Vert_{C^0(\bar \Omega)}+ \Vert \nabla u \Vert_{C^0(\bar \Omega)} + \Vert \Delta u \Vert_{L^{\infty}(\Omega)}< + \infty.
$$

By the Calderon-Zygmund theory a function $u \in C^{1,\bar 1}(\bar \Omega)$ satisfies $u \in W^{2,p} (\Omega)$ for any $1 \leq p< \infty$ and  then by  Sobolev spaces  theory (Morrey's Lemma), it follows that $ u \in C^{1, \alpha}(\bar{\Omega})$ for any $\alpha \in]0,1[$ (see \cite{GT98}).

We do not know if the  eigenfunction $u_1$ in our theorem satisfies $u_1 \in C^{1,1}( \bar{\Omega})$.

\smallskip

 On the other hand we develop  a general variational approach to the problem of the existence of a first eigenvalue for more general twisted Monge-Ampère operators. Using this approach and our new existence theorem  we are able  to prove a   Rayley quotient type formula for the eigenvalue. This is our second main result. 

\smallskip

\begin{theorem}  \label{thm:variational} Let $\Omega \Subset \C^n$ be a bounded strongly pseudoconvex domain with smooth boundary and $0 < f \in C^{\infty}(\bar \Omega)$ be a smooth positive  function in $\bar \Omega$. 
Let    $(\lambda_1,u_1)$ be  the normalized solution of the eigenvalue problem (\ref{eq1}). Then 
$$
 \lambda_1^n = \frac{\int_\Omega (-u_1) (dd^c u_1)^n}{\int_\Omega (-u_1)^{n +1} d \nu_{f}} = \inf \left\{\frac{\int_\Omega (-\phi) (dd^c \phi)^n}{\int_\Omega (-\phi)^{n +1} d \nu_{f}} ; \phi \in \mathcal E^1(\Omega), \phi \neq 0 \right\},
 $$ 
 where $d \nu_{f} = f^n \omega^n$.
\end{theorem}
Here $\mathcal E^1(\Omega) \subset PSH(\Omega)$ is the convex positive cone of negative plurisubharmonic functions in $\Omega$ with finite Monge-Ampère energy (see Section 5 for the precise definition).

\smallskip

Let us mention that N.D. Koutev and I.P. Ramadanov considered the eigenvalue problem  for a twisted complex Monge-Ampère operator in bounded strongly pseudoconvex domains with a homogenous non degenerate strictly positive right hand side which do not contain the case considered here (see \cite{KR89, KR90}). 
 On the other hand a Rayley quotient type formula for the  eigenvalue of the real Monge-Ampère operator was proved  by K. Tso in \cite{Tso90}  with a different approach.  
 
 \smallskip

The paper is organized as follows. In Section 2, we recall some known results which will be useful for the proofs of the main theorems.

 In Section 3, we prove a priori estimates as well as a new existence result for  some degenerate complex Monge-Ampère equations which will be used later.
In Section 4, we prove Theorem \ref{thm1}. As an application, we  show that the eigenvalue $\lambda_1$ plays the role of a bifurcation parameter for  a family of operators of complex Monge-Ampère type with respect to the maximum principle as in the case of  linear second order  elliptic operators. 
In section 5 we develop a variational approach to the  eigenvalue problem  and prove Theorem \ref{thm:variational} using our new existence theorem.

\smallskip

\smallskip

{\bf Aknowledgements :} 
This is a corrected version of the article published in the JGEA in 2023 (\cite{BZ23}). The changes concern a rectification of the proof of \cite[Proposition 3.2]{BZ23} and its corollary \cite[Corollary 3.3]{BZ23}. The original proofs of these results contained a minor gap, which was brought to our attention by Professor Semyon Alesker. We thank him for sharing this pertinent observation.
We have filled this gap with the help of Chinh H. Lu and submitted an erratum to JGEA (see \cite{BLZ26}). We thank Chin H. Lu for his help.

\section{Preliminaries}
Let us first recall some  results in Pluripotential Theory that will be used in the sequel. The general reference for this material is \cite{GZ17}.
\subsection{The complex Monge-Amp\`ere opertaor }
Let $\Omega \Subset \C^n$ be a bounded domain.
 Recall  that by E. Bedford and B.A. Taylor \cite{BT76,BT82}, the complex Monge-Ampère operator is well defined on the class of locally bounded plurisubharmonic functions on $\Omega$. For such a function $v \in PSH(\Omega) \cap L_{loc}^{\infty} (\Omega)$, $(dd^c v)^n$ is a positive current of bidegree $(n,n)$ on $\Omega$, which will be identified to a Radon measure on $\Omega$. By the classical representation theorem of F. Riesz, this Radon measure extends (uniquely) to a positive Borel measure on $\Omega$ (with locally finite mass), called the complex Monge-Ampère measure of $v$ and still denoted by $(dd^c v)^n$. 
Recall that when $v \in PSH (\Omega) \cap C^{2}(\Omega)$  we have 
$$
(dd^c v)^n = \det(v_{j \bar k}) \omega^n.
$$
Moreover the complex Monge-Amp\`ere operator  is continuous under (local) uniform convergence and monotone convergence of sequences in $PSH(\Omega) \cap L_{loc}^{\infty} (\Omega)$  and satisfies the following comparison principle. 

\begin{proposition} \cite{BT76} \label{prop:CP}. Let $u,v \in PSH (\Omega) \cap L^{\infty}(\Omega)$ such that  for $\zeta \in \partial \Omega$ $\liminf_{ z \to \zeta} (u - v) \geq 0$ and $(dd^c u)^n \leq (dd^c v)^n $ in the sense of currents on $\Omega$.
Then $ u \geq v$ in $\Omega$.
\end{proposition}

 Following  \cite{BT82} we define the Monge-Amp\`ere capacity as follows :  for any Borel set $B \subset \Omega$, set
$$
\text{Cap}_\Omega (B) := \sup \left\{ \int_B (dd^c v)^n \ ; \, v \in PSH(\Omega), -1 \leq v \leq 0 \right\}\cdot
$$

This capacity plays an important role in Pluripotential Theory (see \cite{BT82}).
 Let $D > 0$ be the diameter of $\Omega$ and $R := D\slash 2$ and $B(a,R)$ a ball containing $\Omega$. Then using the test function  $v (z) := (R^{-2}\vert z- a\vert^2 - 1)$ in the definition of the capacity, we easily see that for any Borel set $B \subset \Omega$, 
 \begin{equation} \label{eq:Volcap}
 \text{Vol} (B) \leq R^{2n} \text{Cap}_\Omega (B),
 \end{equation}
 where $ \text{Vol} (B) := \int_B d V =  \int_B (dd^c \vert z\vert^2)^n $ is up to a numerical multiplicative constant the volume of $B$ in $\C^n$. 
\subsection{The CKNS theorem}
Let $\Omega \Subset \C^n$ be a strongly pseudoconvex domain with $C^2$ boundary i.e. $\Omega $ admits a defining function $\rho \in C^2(\bar \Omega)$ which is strictly plurisubharmonic in a neighborhood of $\bar \Omega$ and satisfies $\Vert \nabla \rho \Vert > 0$ in $\partial \Omega$.
We consider  the following Dirichlet problem for the complex Monge-Ampère operator:
\begin{equation}\label{eq:DP}
\left\lbrace
\begin{array}{lcr}
(dd^c u)^n=\psi(\cdot,u) \, \omega^n & \textnormal{in} & \Omega,\\
u=\varphi& \textnormal{ in} & \partial \Omega,
\end{array}
\right.
\end{equation}
where $\omega = dd^c \vert z\vert^2$, $\psi :\bar \Omega \times \R \longrightarrow  \R^+$ and $\varphi : \partial \Omega \longrightarrow \R$ are smooth functions.

We will need the following fundamental theorem due to L. Caffarelli, J.J. Kohn, L. Nirenberg and J. Spruck (see \cite{CKNS85}).

\begin{theorem} \label{thm:CKNS}  Let $\Omega \Subset \C^n$ be a bounded strongly pseudoconvex domain with smooth boundary. Let $\varphi \in C^{\infty} (\partial \Omega)$ and $\psi \in C^{\infty} (\bar \Omega \times \R)$  such that $\partial_t \psi (z,t) \geq 0$ and $\psi  > 0$ in  $\bar \Omega \times \R$. 

Then the Dirichlet problem \eqref{eq:DP} admits a unique solution $u \in PSH (\Omega) \cap C^{\infty} (\bar \Omega)$.
\end{theorem}

\subsection{The subsolution theorem of B. Guan}
Let us recall that under the condition $\frac{\partial \psi}{\partial u}\geq 0$, Caffarelli-Kohn-Nirenberg-Spruck
used  the continuity method and a priori estimates to prove the existence of a unique solution to the Dirichlet problem \eqref{eq:DP}.

In the general case, the continuity method do not apply. However under the strong assumption that $\psi$ admits a positive lower bound,  B. Guan (\cite{Guan98})  was able to extend the CKNS result to a more general situation  using a priori estimates  and the more involved method of topological degree as in the real case (see \cite{CNS84}). 

Let us state the theorem of B. Guan \cite[Theorem 1.1]{Guan98})  which we will  use only in our application (see Theorem \ref{thm:Uniqueness}). 
\begin{theorem}\cite{Guan98} \label{thm:GuanB}
Let $\Omega \Subset \C^n$ be a bounded domain with smooth boundary, $\varphi \in C^{\infty} (\partial \Omega)$. Assume that $\psi \in C^{\infty }(\bar \Omega \times \R)$,  $\psi  > 0$ in $\bar \Omega \times \R$,   and the problem \eqref{eq:DP} admits a strictly plurisubharmonic subsolution $\underline u \in C^{2}(\bar \Omega)$. 

Then the problem \eqref{eq:DP} admits a smooth solution $u \in  PSH  (\Omega) \cap C^{\infty}(\bar \Omega)$ such that $u \geq \underline u$ in $\Omega$.
\end{theorem} 
This theorem cannot be applied to prove Theorem \ref{thm1} because in this case, the right hand side  $(- \lambda u)^n$  degenerates at the boundary as $u=0$ in $\partial \Omega$.
We will prove a new theorem which deals with this case (see Theorem \ref{thm:new}).

\subsection{The eigenvalue problem for linear elliptic  operators}

To prove Theorem \ref{thm1}, we need to recall  few results from the  theory of linear elliptic operators of second order. The link to the complex Monge-Ampère operator is provided by 
Gaveau's formula \cite{Gav77} in  linear algebra : if $b$ is a positive Hermitian  $n \times n$ matrix, we have 

\begin{equation}\label{eq215}
	(\det b)^\frac{1}{n}=\frac{1}{n}\inf\{tr(a\cdot b): ~a \in \mathcal{H}_n \},
\end{equation} 
where $\mathcal{H}_n$ be the set of Hermitian positive $ n \times n$ matrices with $\det a \geq 1$.

This formula makes it possible to reduce the equation \eqref{eq1} to a Hamilton-Jacobi-Bellman (HJB) type equation.

Let $\mathcal{A}(\Omega)$ be the set of  positive  Hermitian $n \times n$ matrices $a=\{a_{j\bar k}\}_{1\leq j,k\leq n}$ with continuous coefficients in $\Omega$ and bounded in $\Omega$  such that $\text{det} \,  a \geq 1$.

To each matrix $a=\{a_{j\bar k}\}_{1\leq j,k\leq n}\in \mathcal{A}(\Omega)$, we associate the  second-order linear differential operator
\begin{equation}\label{eq216}
	L_a=\frac{1}{n}\sum_{j,k=1}^na_{j\bar k}\frac{\partial ^2}{\partial z_j\partial \bar z_k}.
\end{equation}
The operator $L_a$ is elliptic and then staisfies the maximum principle.
 Moreover Gaveau's formula implies that if $u\in C^2(\Omega)$, we have pointwise in $\Omega$,

\begin{equation}\label{eq217}
	(\det u_{j\bar k})^\frac{1}{n}=\inf\Big\{L_au  \, ; \, ~ a \in \mathcal{A}(\Omega)\Big\}.
\end{equation}

Following the idea of Lions in the real case \cite{Lions86}, we will show the existence of an eigenvalue of the complex Monge-Amp\`ere operator using the first eigenvalues of the linear operators $L_a$ which define it. 

We will need the following classical result $($see \cite[p. 335-340]{Eva10} $)$:

\begin{lem}\label{lem1}
	There exists a unique pair $(\gamma_1,\phi_1)=(\gamma_1(a), \phi_1(a))$ satisfying $\gamma_1>0, \phi_1\in C^2(\bar \Omega)$, $\phi_1 > 0$  and $\Vert \phi_1\Vert_{C^0(\bar \Omega)}=1$ such that $(\gamma_1,\phi_1)$ is a  solution of the Dirichlet problem
	\begin{equation}\label{eq218}
	\left\lbrace
	\begin{array}{lcr}
	L_a\phi_1=-\gamma_1\phi_1f& \textnormal{in}& \Omega,\\
	\phi_1=0 & \textnormal{on}& \partial\Omega\\
	\phi_1<0& \textnormal{in}& \Omega.
	\end{array}
	\right.
	\end{equation}
\end{lem}
The number $\gamma_1 (a) = \gamma_1(L_a,\Omega)$ is called the first eigenvalue of the operator $-L_a$ and $\phi_1$ is an associated eigenfunction.

We will use the following results to prove uniqueness in Theorem \ref{thm1}.

\begin{prop}\cite{BNV94}\label{pro3} Let $ a \in \mathcal A(\Omega)$ and $\gamma_1 := \gamma_1 (a)$. Then 

1. $\gamma_1$ is given by the following formula :
\begin{equation}\label{eq219}
	\gamma_1(a)=\gamma_1(L_a,\Omega)=\sup\{\gamma \geq 0:~~\exists   \phi < 0, L_a\phi +\gamma \phi f \geq 0 \},
\end{equation}
where $\phi  \in C^{2} (\Omega)$ and $\phi = 0$ in $\partial \Omega$. 

2.  If $\phi \in C^2 (\Omega)$ is bounded from above in $\Omega$, $\limsup_{z\to \zeta}\phi(z)\leq 0$ for all $\zeta \in \partial \Omega$ and satisfies $L_a\phi+\gamma_1 \phi f\geq 0$ in $\Omega$,  then there exists a constant $\theta \in \R$ such that $\phi=\theta \phi_1.$ 
\end{prop}

All these results are stated in \cite{BNV94} in the case $f \equiv 1$ in $\Omega$, but they are still valid for a general positive density $f > 0$ (see \cite{NP92}).

To apply this result in our context, we will need the following observation.

\begin{lem}\label{lem2}
	Let  $u,v \in PSH(\Omega) \cap C^{1,\bar 1}(\bar \Omega) \cap C^2(\Omega).$ Then there exist a matrix $a\in \mathcal{A}(\Omega)$ such that 
 $$
   [\det(u_{j\bar k}) ]^\frac{1}{n}- [\det(v_{j\bar k}) ]^\frac{1}{n} = L_a(u-v).
 $$
\end{lem}

\begin{proof}	
 Recall that the function $x \mapsto \det x$ is a $C^1$-function on the space $\mathcal{M}_n$ of square matrices of order $n$  
 and its differential is given for $x\in \mathcal{M}_n$ and $\xi \in T_x\mathcal{M}_n\simeq \mathcal{M }_n$, by the formula 
	
 \begin{equation}\label{eq220}
		D_\xi (\det) (x) = D(\det)(x)\cdot \xi=tr(\tilde{ x}^\tau \xi),
	\end{equation}
where $\tilde x  $ be the comatrix of $x$ and $\tilde{ x}^\tau$ its transpose.

Set $g(t):=[\det(tu_{j\bar k}+(1-t) v_{j\bar k})]^\frac{1}{n}$ for $0 \leq t \leq 1$. 
 Then $g$ is differentiable in $[0,1]$ and therefore 	
	$$
	[\det(u_{j\bar k}) ]^\frac{1}{n}- [\det(v_{j\bar k}) ]^\frac{1}{n} = g(1)-g(0)=\int_{0}^{1}g'(t)dt.
	$$	
	Applying the formula \eqref{eq220} of the differential of $\det $ at the point 	
$x(t):=t(u_{j\bar k})+ (1-t)(v_{j\bar k})$ with $\xi:=(u_{j\bar k})-(v_{j\bar k}),$ we obtain	
 $$
 g'(t):=\frac{1}{n}[\det (x(t))]^{\frac{1}{n}-1}tr(\tilde x(t)^\tau \xi).
 $$
	
 It follows by setting $w:=u-v$ that	
 $$
 [\det(u_{j\bar k}) ]^\frac{1}{n}- [\det(v_{j\bar k}) ]^\frac{1}{n} =\frac{1}{n}\sum_{j,k=1}^n a_{j \bar k} w_{j \bar k}=L_a w,
 $$
 where
 $$
 a_{j \bar k}:=\int_{0}^1[\det (x(t))]^{\frac{1}{n}-1}\tilde{ x}_{j\bar k}(t) dt.
 $$
 It remains to show that $a \in \mathcal{A}(\Omega)$.
	
	Indeed, $\det^\frac{1}{n}$ is a concave function on the convex set $\mathcal H_n$ of  positive Hermitian $n \times n$ matrices $($this result follows from Gaveau's formula$)$. Therefore
 \begin{align*}
	 (\det a)^\frac{1}{n} &\geq \int_{0}^1\Big[\det \Big( [\det x(t)]^{\frac{1}{n}-1}\tilde{ x}(t)\Big)\Big]^\frac{1}{n}dt\\&=
 \int_{0}^1[\det x(t)]^{\frac{1}{n}-1}[\det \tilde{ x}(t) ]^\frac{1}{n}dt.
\end{align*}
	
	It remains to calculate $\det \tilde x (t)$. This is done by remembering the formula for the inverse of a matrix $x\in \mathcal{H}_n$: $x^{-1}= \frac{1}{\det x}\tilde x ^\tau$ and therefore $\det \tilde x =(\det x)^{n-1}.$ Thus $\det a \geq 1$ and $a \in \mathcal{A}(\Omega).$	
\end{proof}


\section{A new existence theorem}

To prove the main results stated in the introduction, we need to prove a new theorem on the existence of a solution to some special degenerate complex Monge-Ampère equations.

Let $\Omega \Subset \C^n$ be a bounded domain with smooth boundary and $\psi: \bar \Omega\times \R \to \R^+$ be a  smooth function on $\bar \Omega \times \R$, $\psi \geq 0$.

Consider the Dirichlet problem for the complex Monge-Ampère operator:
\begin{equation}\label{eq:MA0}
\left\lbrace
\begin{array}{lcr}
(dd^c u)^n=\psi(\cdot,u)\omega^n & \textnormal{in} & \Omega,\\
u=0& \textnormal{ in} & \partial \Omega,
\end{array}
\right.
\end{equation}
where $u \in \mathcal P(\Omega):= PSH(\Omega) \cap C^2(\Omega) \cap C^0(\bar \Omega)$ is the unknown function.


 


As we already said, the  solvability of this  problem when $\Omega \Subset \C^n$ is a bounded strongly pseudoconvex domains in $\C^n$, was established by L. Caffarelli, J.J. Kohn, L. Nirenberg and J. Spruck \cite{CKNS85}  under the assumption  that $\frac{\partial \psi}{\partial u}\geq 0$ and some conditions on the way $\psi$ degenerates near the boundary.  Later this result was extended by B. Guan \cite{Guan98} and B. Guan and Q. Li \cite{GL10} to a more general situation  in the non degenerate case on a general domain, assuming  the existence of a strictly plurisubharmonic subsolution. 

However these results do not apply in our case.
Using ideas from \cite{CKNS85}, we are able to prove a new existence theorem  for such equations under the assumption of the existence of  a subsolution and a supersolution by a fixed point method.

\smallskip
 
We say that $\underline u \in \mathcal P(\Omega)$ is a subsolution of \eqref{eq:MA0} if it satisfies 

\begin{equation}\label{eq3}
\left\lbrace
\begin{array}{lcr}
(dd^c \underline u)^n \geq \psi(z,\underline u) \omega^n & \textnormal{on} & \Omega,\\
\underline u=0& \textnormal{ on} & \partial \Omega.
\end{array}
\right.
\end{equation}\mbox{}
Moreover the subsolution  $\underline u$ is said to be a {\it strict subsolution} if 
\begin{equation} \label{eq:StrictSubsolution}
(dd^c \underline u)^n \geq (\psi(z,\underline u) + \epsilon_0)   \omega^n, \, \, \textnormal{ on} \, \,  \, \Omega.
\end{equation}

Also we say that $\bar u \in \mathcal P(\Omega)$ is a supersolution of \eqref{eq:MA0}, if it satisfies
\begin{equation}\label{eq4}
\left\lbrace
\begin{array}{lcr}
(dd^c \bar u)^n\leq\psi(z,\bar u)\omega^n & \textnormal{on} & \Omega,\\
\bar u=0& \textnormal{ on} & \partial \Omega.
\end{array} \right.
\end{equation}
The differential inequality here is understood in the sense of currents on $\Omega$.

\smallskip

Here is our  new result.

\begin{theorem}\label{thm:new}
 Let $\Omega \Subset \C^n$ be a bounded  strongly pseudoconvex domain with smooth boundary,  $0 \leq \psi^\frac{1}{n}  \in C^{\infty} (\bar{\Omega} \times \R)$  and  $\frac{\partial \psi}{\partial u}\leq 0$ in $\Omega \times ]-\infty, 0]$. 
  
  Assume  the following conditions:  
  
  $(1)$ If $\psi > 0$ in $\bar{\Omega} \times ]-\infty,0]$ , the problem \eqref{eq:MA0} admits a subsolution $\underline u \in PSH(\Omega)\cap C^{\infty}(\bar \Omega)$;
  
  \smallskip
  
  $(2)$ If $\psi \geq 0$ in $\bar{\Omega} \times  ]-\infty,0]$ and $\psi > 0$ in $\Omega \times ]-\infty,0[$, the Dirichlet  problem \eqref{eq:MA0} admits a strict subsolution $\underline u \in PSH(\Omega)\cap C^{\infty}(\bar \Omega)$;
  
  \smallskip
  
   $(3)$   the Dirichlet problem \eqref{eq:MA0} admits a supersolution $\bar u \in PSH(\Omega)\cap C^0(\bar \Omega)$ such that $\underline u \leq \bar u < 0$ on $\Omega$.
   
   \smallskip
   
 Then the Dirichlet problem \eqref{eq:MA0} admits a solution $u \in PSH(\Omega) \cap C^{\infty}(\Omega)  \cap C^{1,\bar 1}( \bar \Omega)$ such that $\underline u \leq u \leq \bar u$  on $\bar \Omega.$   
  \end{theorem}

\smallskip
   
 The proof of Theorem \ref{thm:new} will be given below. The idea of the proof is as follows.  When $\psi > 0$ in $\bar \Omega \times ]-\infty,0]$, by Theorem \ref{thm:CKNS}, for any $v \in   PSH(\Omega)\cap C^{\infty}(\bar \Omega)$, the following Dirichlet problem 
 \begin{equation}\label{eq:MA}
\left\lbrace
\begin{array}{lcr}
(dd^c u)^n=\psi(\cdot,v)\omega^n & \textnormal{in} & \Omega,\\
u=0& \textnormal{ in} & \partial \Omega,
\end{array}
\right.
\end{equation}
 admits a unique solution $u = T(v) \in   PSH(\Omega)\cap C^{\infty}(\bar \Omega)$.
 
 The solution of the Dirichlet problem  \eqref{eq:MA0} is a fixed point for the inverse Monge-Amp\`ere operator $T$. We will use an iterative method to find a fixed point of $T$.   We then need to establish  a priori estimates to control  the iteration process. This  will be done  in the next three subsections.

 \subsection*{A priori estimates}
 In this subsection, we assume the following conditions. 
\begin{enumerate}
 \item $\Omega$ is a smooth bounded strongly pseudoconvex domain in $\mathbb{C}^n$, $\rho$ is a smooth defining function of $\Omega$ such that $dd^c \rho \geq \omega =dd^c |z|^2$; 
	
\item  $\psi $  is smooth  on $\bar \Omega \times ]-\infty, 0]$ and $\psi > 0$ in $\Omega \times ]-\infty, 0[$;
\item  $\phi \in C^{0}(\bar \Omega)$  is a given function satisfying   $ \phi < 0$ in $\Omega$;
\item $v \in PSH(\Omega)  \cap C^1(\bar \Omega)$ is a given function such that $v \leq \phi$ in $\Omega$.
\end{enumerate}
We consider the following Dirichlet problem 
 \begin{equation}\label{eq:MA}
\left\lbrace
\begin{array}{lcr}
(dd^c u)^n=\psi(\cdot,v) \, \omega^n & \textnormal{in} & \Omega,\\
u=0& \textnormal{on} & \partial \Omega.
\end{array}
\right.
\end{equation} 

\subsection{The gradient a priori estimate} 

\begin{proposition} \label{prop:GradientEstimate} Let $u \in PSH(\Omega) \cap C^{3}(\Omega) \cap C^1(\bar \Omega)$ be a solution to the Dirichlet problem \eqref{eq:MA}. 

Then we have the following estimate 
 \begin{equation} \label{eq:GradientEstimate1}
 \sup_{\bar \Omega} |\nabla u|^2 \leq \sup_{\partial \Omega} |\nabla u|^2 +\frac{\sup_{\bar \Omega} |\nabla v|^2}{2} +C_1,
 \end{equation}
 where $C_1$ depends on  $L:= \Vert \nabla \rho\Vert_{C^0(\bar \Omega)}$, $C_0 := \sup_{\bar \Omega} (| u |+|\rho|)$, $M_0 := \sup_{\bar \Omega} \vert v\vert$,  $K_1 :=\Vert \psi^{1 \slash n}\Vert_{C^1(\bar \Omega \times [-M_0,0])}$ and $m_0 := \min_{\bar{\Omega}_0 \times [-M_0 , -A_0]} \psi^{1\slash n}> 0$, where $A_0 > 0$ and $\Omega_0 \Subset \Omega$ depend only on $C_0$ and the $C^{0}(\bar \Omega_0)$-norm of $\phi$.
\end{proposition}
This result is stated in  \cite[Proposition 3.2]{BZ23}. However its proof contains a gap. We will give a correct proof here following closely the argument of \cite[Proposition 3.2]{BZ23}. The only modification concerns the function $G$, defined by \eqref{eq:testfunction} below, where we replace the function $B\rho$ in the original proof by the function $e^{B \rho}$ (see \cite{BLZ26}).

\begin{proof}
We set $\beta= |\nabla u|^2= \sum_p |u_p|^2$, where $u_p$ stands for the partial derivative of $u$ in the direction $z_p$. 

We consider the function 
\begin{equation} \label{eq:testfunction}
G := \log \beta + \frac{u^2}{2}+  e^{B\rho},
\end{equation}
defined and upper semicontinuous in $\bar \Omega$. Here $B$ is a positive constant to be specified later. 
The maximum of $G$ over $\bar \Omega$ is attained at some $z_0 \in \bar \Omega$.

 If $\beta(z_0)\leq 1$, then  $G(z_0)\leq C_0^2 \slash 2 + 1$, hence for any $z \in \Omega$, $G(z) \leq C_0^2  + 1$, which yields the upper bound  $\beta(z) \leq  e^{G(z)} \leq e^{C_0^2 + 1}$ in $\Omega$. 

We can thus assume that $\beta(z_0)\geq 1$. 
If $z_0 \in \partial \Omega$, then for any $z \in \Omega$,
$$
\beta (z) \leq e^{G(z)} \leq e^{G(z_0)} \leq \beta(z_0)  e^{C^2_0  + 1}. 
$$
Hence  $\beta (z) \leq C_1 \max_{\partial \Omega} \beta$ and the estimate \eqref{eq:GradientEstimate1} is satisfied  with any fixed constant  $C_1 \geq e^{C^2_0 + 1}$. 

We can thus assume that $z_0 \in \Omega$ and  $\beta(z_0)\geq 1$. From now, our computations will be done at $z_0 \in \Omega$ and we can assume that the matrix $(u_{i\bar j})$ is diagonal. By the maximum principle we have, for any $1\leq p, q \leq n$, 
\begin{flalign}
	0 &= G_p=\frac{\beta_p}{\beta} + uu_p + B e^{B\rho}\rho_{p}, \nonumber \\
	0 &=G_{\bar q}=\frac{\beta_{\bar q}}{\beta} + uu_{\bar q} + Be^{B\rho}\rho_{\bar q},
\label{eq: gradient first derivative}
\end{flalign}
and we also have

\begin{flalign}
	0 & \geq \sum_{p=1}^n u^{p\bar{p}} G_{p\bar{p}}  
	= \sum_{p=1}^n u^{p\bar{p}} \left ( \frac{\beta_{p\bar{p}}}{\beta} - \frac{|\beta_p|^2}{\beta^2} +|u_p|^2 +u u_{p\bar p}  + B e^{B\rho} \rho_{p \bar p} + B^2 e^{B \rho} \vert \rho_p\vert^2 \right) \nonumber  \\
	&   = \sum_{p=1}^n u^{p\bar{p}} \left ( \frac{\beta_{p\bar{p}}}{\beta} - |u u_p+ Be^{B \rho} \rho_p|^2  + B^2 e^{B\rho}  \vert \rho_p\vert^2 + B e^{B\rho} \rho_{p \bar p} + |u_p|^2  \right) + n u. \label{eq: gradient second derivative 1}
\end{flalign}
 We next compute 
\begin{flalign}
\sum_{p=1}^n u^{p\bar{p}} \beta_{p\bar{p}} &=\sum_{p=1}^n u^{p\bar{p}} \sum_{j=1}^n (u_j u_{\bar j})_{p\bar p} =  \sum_{p=1,j=1}^n u^{p\bar{p}} \left(2 Re(u_{\bar j}u_{jp\bar{p}}) + |u_{jp}|^2 + |u_{j\bar{p}}|^2\right) \nonumber \\
& = \sum_{1\leq p,j\leq n} u^{p\bar{p}} \left(2 Re(u_{\bar j}u_{jp\bar{p}}) + |u_{jp}|^2\right)+ \sum u_{p \bar p}. \label{eq: gradient second derivative 2}
\end{flalign}
Setting $g (z) :=  \log \psi (z,v(z)), z \in  \Omega$, and differentiating the Monge-Amp\`ere equation 
\[
\log \det(u_{i\bar{j}}) =  \log \psi (\cdot,v) = g,
\]
we obtain 
\[
\sum_{p=1}^n u^{p\bar p} u_{p\bar p j}  =  g_j.
\]
Thus, Cauchy-Schwarz inequality gives 
\begin{flalign}  \label{eq: gradient mixed derivative 1}
\sum_{p=1}^n u^{p\bar{p}} \beta_{p\bar{p}} & = \sum u_{p \bar p} + \sum_{1\leq p,j\leq n} u^{p\bar{p}} |u_{jp}|^2+ 2Re\left (\sum_{1\leq j\leq n} g_j u_{\bar j}\right)  \nonumber \\ 
& \geq \sum_{1\leq p,j\leq n} u^{p\bar{p}} |u_{jp}|^2 -  2 |\nabla g| \sqrt{\beta}.
\end{flalign}
We next estimate the first term on the right-hand side of \eqref{eq: gradient mixed derivative 1}. 

By  \eqref{eq: gradient first derivative}, and the fact that 
$(u_{p \bar q})$ is diagonal, we have

\[
\sum_{1\leq j\leq n} u_{jp}u_{\bar j} + u_p u_{p\bar p} = -\beta (u u_p + B e^{B \rho}\rho_p),
\]
thus Cauchy-Schwarz inequality gives
\[
|\beta (uu_p +B e^{B \rho} \rho_p) + u_p u_{p\bar p} |^2  \leq \beta \sum_{j} |u_{jp}|^2.
\]
Therefore, 
\begin{flalign}
	\sum_{j,p} \frac{u^{p\bar p}|u_{jp}|^2}{\beta} & \geq \beta^{-2} \sum_{p} u^{p\bar p} |\beta (uu_p + 
B e^{B \rho} \rho_p) + u_p u_{p\bar p} |^2\nonumber \\
	& \geq  \sum_{1\leq p\leq n} u^{p\bar p} |uu_p +  B e^{B \rho}\rho_p|^2 - 2 \beta^{-1} \sum_{1\leq p\leq n} |u u_p +  Be^{B\rho}\rho_p| \vert u_p\vert \nonumber \\
	&\geq   \sum_{1\leq p\leq n} u^{p\bar p} |u u_p +  B e^{B \rho} \rho_p|^2 - 2 C_0 - 2 B L \beta^{-1\slash 2}. \label{eq: gradient mixed derivative 2}
\end{flalign}
From \eqref{eq: gradient mixed derivative 1} and \eqref{eq: gradient mixed derivative 2} and $\beta (z_0) \geq 1$, we obtain 
\begin{flalign}
	\beta^{-1}\sum_{1\leq p \leq n} u^{p\bar{p}} \beta_{p\bar{p}} &\geq  \sum_{1\leq p\leq n} u^{p\bar p} |u u_p + B e^{B \rho} \rho_p|^2 - 2 \beta^{-1/2} |\nabla g|   -A_1,
\end{flalign} 
where $A_1 :=  2 C_0 +  2BL$. 

Observe that $g (z) = \log \psi (z,v(z))  = n \log \psi^{1\slash n} (z,v(z)))$, hence 
$$
\nabla g =  n\left[\nabla \psi^{1\slash n} (\cdot,v) + (\psi^{1\slash n})_t (\cdot,v) \nabla v\right] \psi^{-1 \slash n} (\cdot,v),
$$
and then 
$$
\vert \nabla g \vert\leq  n K_1 \left(1 + \vert \nabla v \vert\right)  \psi^{-1 \slash n} (\cdot,v).
$$

From  \eqref{eq: gradient second derivative 1} we thus get 
 \begin{flalign} \label{eq: gradient last step}
 0 & \geq  B e^{B\rho} \sum_{p=1}^n u^{p\bar p} - 2n K_1 \beta^{-1/2} (1+|\nabla v|) \psi^{-1\slash n} (\cdot,v) - A_2 +  \sum_{1\leq p\leq n}  
 u^{p\bar p}|u_p|^2,  
 \end{flalign}
 where $A_2 := A_1 + n C_0 $.

\smallskip

We choose $B > 0$ so that $B e^{-1} = 1 + 4 K_1 e^{C_0^2 + 2}$ and set 
\[
B_1:=  e^{-C_0^2  -2}\sup_{\bar \Omega} (1 + |\nabla v|^2). 
\]
If $\beta(z_0)\leq \max(B_1,1)$ then for any $z \in \Omega$,
\[
G(z)\leq G(z_0)\leq \max(\log B_1,0)+ \frac{C_0^2}{2} +1  \leq \max(\log (1 + \sup_{\bar \Omega} |\nabla v|^2) -1, C_0^2+1),
\]
 hence for any $z \in \Omega$,
 \[
 \log \beta(z)\leq G(z) - \frac{u(z)^2}{2} - e^{B \rho(z)} \leq \max(\log (1 +  \sup_{\bar \Omega} |\nabla v|^2) -1, C_0^2+1),
 \] 
 and the  estimate \eqref{eq:GradientEstimate1} follows with any fixed constant $C_1 \geq e^{C_0^2 + 1}$.  
 
 \bigskip
 
We thus assume in the sequel that $\beta(z_0)\geq \max(B_1,1)$.

\smallskip
\smallskip

Recalling $\Pi_{1 \leq p \leq n} u^{p \bar p} = \psi^{-1}(\cdot,v)$ and applying the arithmetic-geometric inequality, we then have
 \begin{eqnarray} \label{eq:Ineq1}
 2n K_1 \beta^{-1/2} (1+|\nabla v|)\psi^{-1\slash n}(\cdot,v) &\leq &  4n K_1 e^{C_0^2 + 1} \psi^{-1\slash n}(\cdot,v) \nonumber \\
 & \leq & 4 K_1 e^{C_0^2 + 1} \sum_{p=1}^n u^{p\bar p}. 
 \end{eqnarray}
 Then  from \eqref{eq: gradient last step} and \eqref{eq:Ineq1}, it follows that
 $$
 0  \geq  (B e^{B\rho} -  4 K_1 e^{C_0^2 + 1}) \sum_{p=1}^n u^{p\bar p}   - A_2 +  \sum_{1\leq p\leq n}  
 u^{p\bar p}|u_p|^2.
 $$
 
 We consider  two cases.   
 
 \smallskip
  \smallskip
 
 \underline{{\bf Case 1: $\rho(z_0)\geq -B^{-1}$}}. 
 \smallskip
  \smallskip
  
 We then  have 
$$
B e^{B\rho(z_0)} \geq B e^{-1} = 1 + 4 K_1 e^{C_0^2 + 2}.
$$
 
Then  from \eqref{eq: gradient last step} and \eqref{eq:Ineq1}, it follows that
 \begin{equation}  \label{eq: gradient last step2}
	0  \geq   \sum_{p=1}^n u^{p\bar p}  - A_2 + \sum_{1\leq p\leq n} u^{p\bar p}|u_p|^2. 
\end{equation}
   Hence $\sum_{1\leq p\leq n}u^{p\bar p}\leq A_2$.
   
   Now recall the following elementary inequality for $0< \lambda_1 \leq  \cdots \leq \lambda_n$,
\begin{equation} \label{eq:ElementaryIneq}
\sum_{1 \leq j \leq n} \lambda_j \leq n \left(\Pi_{1 \leq j \leq n} \lambda_j\right)  \left(\sum_{1 \leq j \leq n} \lambda_j^{-1}\right)^{n-1}.
\end{equation} 
Applying this inequality with $\lambda_j = u_{j \bar j}$, we deduce that $\sum_{p=1}^n u_{p\bar p} \leq A_3:= n A_2^{n-1} K_1$. Since $0\leq u_{p\bar p}$, it follows that $u_{p\bar p}\leq A_3$, hence $u^{p\bar p}\geq A_3^{-1}$. Putting this into \eqref{eq: gradient last step2} we obtain $\beta (z_0) \leq A_4 := A_2 A_3$, where $A_4 >0$ is a uniform contant independent of $v$. Hence for any $z\in  \Omega$, 
$$
\beta(z) \leq e^{G(z)} \leq e^{G(z_0)} \leq \beta(z_0) e^{C_0^2 + 1 \leq A_4 e^{C_0^2 + 1}},
$$ 
which yields the  inequality \eqref{eq:GradientEstimate1} 
with any fixed constant $C_1 \geq A_4 e^{C_0^2 + 1}$.

\smallskip

\underline{{\bf Case 2: $\rho(z_0) < -B^{-1}$}}. 
\smallskip

Observe that  $\Omega_0 := \{ z \in \Omega ; \rho(z) <  -B^{-1}\} \Subset \Omega$,  $z_0 \in  \Omega_0$. By continuity of $\phi$ and the fact that $\phi < 0$ on $\Omega$, it follows 
that $ \phi (z_0) \leq - A_0 := \max_{\bar{\Omega}_0} \phi < 0$. Hence $v (z_0) \leq \phi (z_0)  \leq -A_0 < 0$.

Since  $z_0 \in  \Omega_0$, $- M_0 \leq v(z_0) \leq -A_0$ and  $\psi > 0$ in $\Omega \times ]-M_0,0[$, it follows that 
$$
\psi^{1\slash n} (z_0,v(z_0)) \geq m_0 := \min_{\bar{\Omega}_0 \times [-M_0,-A_0]} \psi^{1\slash n}  (z,t) > 0.
$$ 

Since $\beta(z_0)\geq e^{-C_0^2  - 2} \sup_{\bar \Omega}(1 +|\nabla v|^2)$,  we have 
$$
\beta^{-1 \slash 2} (z_0) \vert \nabla v(z_0)\vert \psi^{-1 \slash n} (z_0,v(z_0)) \leq m_0^{-1}  e^{C_0^2 \slash 2 - 1} =: A_5,
$$ 
and then from \eqref{eq: gradient last step}, it follows that  
\begin{eqnarray*}
0 \geq  B e^{-BC_0}  \sum_{p=1}^n u^{p\bar p} -A_6 + \sum_{p=1}^n u^{p\bar p} |u_p|^2 \geq  \sum_{p=1}^n u^{p\bar p} -A_6 + \sum_{p=1}^n u^{p\bar p} |u_p|^2,
\end{eqnarray*}
 where $A_6 := A_2 + 2 n K_1 A_5$.  

Since $B  \geq 1$, this implies that $ \sum_{p=1}^n u^{p\bar p} \leq A_6 e^{B C_0}$ and yields the desired estimate in the same way as before with any fixed  constant $C_1 \geq A_2 A_6 e^{B C_0} e^{C_0^2 + 1}$. 

Finally the gradient estimate \eqref{eq:GradientEstimate1} holds with the constant 
$$
C_1 :=  A_2 A_6 e^{B C_0} e^{C_0^2 + 1}.
$$
\end{proof}
 
 
 \smallskip

\smallskip

We can improve the previous a priori estimate for a fixed point of the operator $T$, giving a correct proof of \cite[Corollary 3.3]{BZ23}.
 
\begin{proposition} \label{coro:GradientEstimate2} Let $u \in PSH(\Omega) \cap C^{3}(\Omega) \cap C^1(\bar \Omega)$ be a  solution to the complex Monge-Amp\`ere equation 
$$
(dd^c u)^n = \psi (\cdot,u) \omega^n
$$ 
such that $u \leq 0$ in $\bar \Omega$.

Then we have the following estimate 
	\[
	\sup_{\bar \Omega} |\nabla u|^2 \leq  C_1 \left(\sup_{\partial \Omega} |\nabla u|^2  + 1\right),
	\]
 where $C_1$ depends on  $L := \Vert \rho\Vert_{C^1(\bar \Omega)}$, $C_0 := \sup_{\bar \Omega} (| u |+|\rho|)$  and  $K_1:=\Vert \psi^{1 \slash n}\Vert_{C^1(\bar \Omega \times [-C_0,0])}$.
 \end{proposition}
 It is important to emphasize  that the constant $C_1$ does not depend on a lower bound of $\psi$ in contrast with the previous estimate.
 
 \smallskip
 
 \begin{proof} We use the same notation  $\beta := \vert \nabla u\vert^2$ and proceed in the same way as in the previous proof except that we consider instead the following function 
\[
G:= \log \beta + \frac{u^2}{2}+ B \rho
\]
defined and upper semicontinuous on $\bar \Omega$, where $B$ is a positive constant to be specified later. 
The maximum of $G$ over $\bar \Omega$ is attained at some $z_0 \in \bar \Omega$.

As before, we can assume that $z_0 \in \Omega$, $\beta(z_0)\geq 1$, and the matrix $(u_{i\bar j} (z_0))$ is diagonal. By the maximum principle we have, for any $1\leq p, q \leq n$, 
\begin{flalign} 
	0 &= G_p=\frac{\beta_p}{\beta} + uu_p + B\rho_{p}
\label{eq: gradient first derivative1.2}
\end{flalign}
and 
\begin{flalign}
	0 & \geq \sum_{p=1}^n u^{p\bar{p}} G_{p\bar{p}}  = \sum_{p=1}^n u^{p\bar{p}} \left ( \frac{\beta_{p\bar{p}}}{\beta} - \frac{|\beta_p|^2}{\beta^2}  + B +|u_p|^2 +uu_{p\bar p} \right) \nonumber  \\
	&   = \sum_{p=1}^n u^{p\bar{p}} \left ( \frac{\beta_{p\bar{p}}}{\beta} - |u u_p+ B\rho_p|^2  + B +  |u_p|^2  \right) + n u. \label{eq: gradient second derivative 1}
\end{flalign}
%
With $f=\log \psi(\cdot, u)$, the computations in the proof of Proposition  
\ref{prop:GradientEstimate} give 
\begin{flalign}  \label{eq: gradient mixed derivative 1}
\sum_{p=1}^n u^{p\bar{p}} \beta_{p\bar{p}} & = \sum u_{p \bar p} + \sum_{1\leq p,j\leq n} u^{p\bar{p}} |u_{jp}|^2+ 2Re\left (\sum_{1\leq j\leq n} f_j u_{\bar j}\right)  \nonumber \\ 
& \geq \sum_{1\leq p,j\leq n} u^{p\bar{p}} |u_{jp}|^2 -  2 |\nabla f| \sqrt{\beta}.
\end{flalign}

By \eqref{eq: gradient first derivative1.2} we have 
\[
\sum_{1\leq j\leq n} u_{jp}u_{\bar j} + u_p u_{p\bar p} = -\beta (u u_p + B\rho_p),
\]
thus Cauchy-Schwarz inequality gives
\[
|\beta (uu_p +B\rho_p) + u_p u_{p\bar p} |^2  \leq \beta \sum_{j} |u_{jp}|^2,
\]
therefore, 
\begin{flalign}
	\sum_{j,p} \frac{u^{p\bar p}|u_{jp}|^2}{\beta} & \geq \beta^{-2} \sum_{p} u^{p\bar p} |\beta (uu_p + B\rho_p) + u_p u_{p\bar p} |^2\nonumber \\
	& \geq  \sum_{1\leq p\leq n} u^{p\bar p} |uu_p +  B\rho_p|^2 - 2 \beta^{-1} \sum_{1\leq p\leq n} |u u_p +  B\rho_p| \vert u_p\vert \nonumber \\
	&\geq   \sum_{1\leq p\leq n} u^{p\bar p} |u u_p +  B\rho_p|^2 - 2 C_0 - 2 B L \beta^{-1\slash 2}, \label{eq: gradient mixed derivative 2}
\end{flalign}
where $L=\sup_{\bar \Omega}|\nabla \rho|$. 
From \eqref{eq: gradient mixed derivative 1} and \eqref{eq: gradient mixed derivative 2} and $\beta (z_0) \geq 1$, we obtain 
\begin{flalign}
	\beta^{-1}\sum_{1\leq p \leq n} u^{p\bar{p}} \beta_{p\bar{p}} &\geq  \sum_{1\leq p\leq n} u^{p\bar p} |u u_p +  B\rho_p|^2 - 2 \beta^{-1/2} |\nabla f|   -A_1,
\end{flalign} 
where $A_1 :=  2 C_0 +  2BL$. 

Arguing as in the proof of Proposition \ref{prop:GradientEstimate} with $u$ in the place of $v$, we obtain 
$\vert \nabla f \vert\leq  n K_1 \left(1 + \vert \nabla u \vert\right)  \psi^{-1 \slash n} (\cdot,u),$
where $K_1$ is an upper bound of the $C^{1}$-norm of $\psi^{1\slash n}$ on $\bar \Omega \times [- C_0,0]$.

Observe also that, by the arithmetic-geometric mean inequality, we have 
\[
\sum_{1\leq p\leq n} u^{p\bar p} \geq   n \psi (\cdot,u)^{-1\slash n}.  
\]
Plugging these estimates into \eqref{eq: gradient second derivative 1} we obtain 
 \begin{flalign} \label{eq: gradient last step}
 0 & \geq  \left(B  -  2 K_1 \beta^{-1/2} (1+|\nabla u|)\right)\sum_{p=1}^n u^{p\bar p}  - A_2 +  \sum_{1\leq p\leq n}  
 u^{p\bar p}|u_p|^2,  
 \end{flalign}
 where $A_2 := A_1 + n C_0 $.
 
 Since $\beta(z_0) \geq 1$, it follows that $\beta(z_0)^{-1/2} (1+|\nabla u(z_0)|)) \leq 2$. Therefore the inequality  \eqref{eq: gradient last step} gives 

\begin{flalign} \label{eq: gradient last last step}
 0  \geq  (B  -  4 K_1) \sum_{p=1}^n u^{p\bar p}  - A_2 +  \sum_{1\leq p\leq n}  
 u^{p\bar p}|u_p|^2. 
 \end{flalign}

If we choose $B=1+ 4 K_1$,  the inequality \eqref{eq: gradient last last step} becomes 
  \begin{flalign} 
 0 & \geq   \sum_{p=1}^n u^{p\bar p}  - A_2 +  \sum_{1\leq p\leq n}  
 u^{p\bar p}|u_p|^2.  
 \end{flalign}
 
  Hence $\sum_{1\leq p\leq n}u^{p\bar p}\leq A_2$, and we can proceed as in Proposition \ref{prop:GradientEstimate} to complete the proof. 
   
\end{proof}

\subsection{Laplacian estimate}

Fix a function $v\in PSH(\Omega)\cap C^{2}(\bar \Omega)$. 

\begin{proposition}\label{prop:LaplaceEstimate} Let $u \in PSH(\Omega) \cap C^{4}(\Omega) \cap 
C^{3} (\bar \Omega)$ be a smooth solution to the complex Monge-Amp\`ere equation \eqref{eq:MA}. 
	
	Then the following estimate holds
	\[
	\sup_{\Omega} \Delta u \leq C_2 \left(1 +  \sup_{\partial \Omega} \Delta u + \frac{1 + \Vert \nabla v\Vert^2 +  \sup_{\Omega} \Delta v}{2}\right), 
	\]
	where $C_2$ depends on $r:= \sup_{\Omega}|z|$, $C_0:=\sup_{\bar \Omega}|u|$, $ M_0 :=\sup_{\bar \Omega} |v|$, $C_1=\sup_{\bar \Omega}(|\nabla u|)$ and  
	$K_2:= \Vert \psi^{1\slash n}\Vert_{C^{1,1}(\bar  \Omega \times [-M_0,0]}$.  
\end{proposition}

\begin{proof}
	We consider the function $H(z) = \log (\Delta u) + b |z|^2$ defined and smooth on ${\Omega}$. Since $H$ is a continuous function on $\bar \Omega$, it attains its maximum over $\bar \Omega$ at some $z_0\in \bar\Omega$. If $z_0\in \partial \Omega$ then for any $z \in \Omega$
 $$
 \Delta u(z) \leq e^{H(z)} \leq e^{H(z_0)} \leq e^{b r^2} \sup_{\partial \Omega} \Delta u,	
 $$
 where $r := \sup_{\Omega} \vert z\vert $.
The estimate \eqref{prop:LaplaceEstimate} follows immediately in this case with any constant $C_2 \geq e^{br^2}$.  
	
 We now assume $z_0\in \Omega$ and do our computations at $z_0$. We can also assume that $u_{i\bar{j}}(z_0)$ is diagonal.  
Since $H$ is $C^2$-smooth in $\Omega$, by the maximum principle we have, for any $1\leq p, q\leq n$, 
\begin{flalign}
	0 &= \frac{\partial H}{\partial z_p}  = H_p= \frac{(\Delta u)_p}{\Delta u} + b \bar{z}_p \nonumber \\
	0 &= \frac{\partial H}{\partial \bar{z}_q}  = H_{\bar{q}}=\frac{(\Delta u)_{\bar q}}{\Delta u} + b z_q,\label{eq: Lap first derivative}
\end{flalign}
and 
\begin{flalign}
	0 &\geq \sum_{p=1}^n u^{p\bar{p}} H_{p\bar{p}}  = \sum_{p=1}^n u^{p\bar{p}} \left ( \frac{(\Delta u)_{p\bar{p}}}{\Delta u} - \frac{|(\Delta u)_p|^2}{(\Delta u)^2}  + b \right) \nonumber \\
	& = \sum_{p=1}^n u^{p\bar{p}} \left ( \frac{(\Delta u)_{p\bar{p}}}{\Delta u} -b^2 r^2  + b \right). \label{eq: Lap second derivative}
\end{flalign}
Differentiating the Monge-Amp\`ere equation $(\det (u_{i\bar j}))^{1/n} = \psi(\cdot,v)^{1/n} =g (\cdot,v)$ twice we obtain, using summation convention, 
\[
n^{-1}g  u^{p\bar q} u_{p\bar q j}  = g_j  + g_v   v_j, 
\]
\[
n^{-2} g  u^{k\bar l} u_{k\bar l \bar j} u^{p\bar q}u_{p\bar q j} + n^{-1}g   (u^{p\bar q})_{\bar j}u_{p\bar q j} + n^{-1}g  u^{p\bar q}u_{p \bar q j \bar j}   =2 Re( g_{vj} v_{\bar j}) + g_v v_{j\bar j} + g_{vv} |v_j|^2+ g_{j\bar j}.
\]
If $D$ is any partial derivative then  $Du^{k \bar q} = -u^{p\bar q} u^{k\bar l} Du_{p\bar l}$, hence, at $z_0$ (using that $u_{i\bar j}$ is diagonal)
\[
Du^{k \bar q} = -u^{q\bar q} u^{k\bar k} Du_{q\bar k}. 
\]
Using again that $(u_{i\bar j})$ is diagonal  we obtain 
\[
n^{-2}g u^{p\bar p} u^{q\bar q}u_{p\bar p \bar j}  u_{q\bar q j} - n^{-1} g u^{p\bar p} u^{q\bar q} |u_{p\bar q j}|^2 + n^{-1}g u^{p\bar p}u_{p \bar p j \bar j}  =2 Re( g_{vj} v_{\bar j}) + g_v v_{j\bar j} +g_{j\bar j},
\]
which is the same as 
\[
n^{-1}g \left| \sum _{p=1}^n u^{p\bar p} u_{p\bar p \bar j} \right|^2-  g \sum_{p,q=1}^nu^{p\bar p} u^{q\bar q} |u_{p\bar q j}|^2 + g \sum_{p=1}^nu^{p\bar p}u_{p \bar p j \bar j}  =n(2 Re( g_{vj} v_{\bar j}) + g_v v_{j\bar j} + g_{vv} |v_j|^2+g_{j\bar j}).
\]
Observe  that, by Cauchy-Schwarz inequality, 
\[
n^{-1} \left| \sum_{p=1}^n u^{p\bar p} u_{p\bar p \bar j} \right|^2-   \sum_{p,q=1}^n u^{p\bar p} u^{q\bar q} |u_{p\bar q j}|^2 \leq 
 n^{-1} \left| \sum_{p=1}^n u^{p\bar p} u_{p\bar p \bar j} \right|^2-   \sum_{p=1}^n|u^{p\bar p} u_{p\bar p j}|^2\leq 0.
\]
We then get 
\[
g \sum_{p=1}^n u^{p\bar p} u_{p\bar p j \bar j} \geq -A_1 (1 + \Vert \nabla v\Vert^2 + \Delta v),   
\]
where $A_1 > 0$ depends only on $\Vert \psi^{1\slash n}\Vert_{C^{1,1} (\bar \Omega \times [-C_0,0]}$.

Plugging this into \eqref{eq: Lap second derivative} we obtain at the point $z_0$,
\[
0\geq (b-b^2r^2 ) \sum_{p=1}^n u^{p\bar p} -\frac{A_1 (1 + \Vert \nabla v\Vert^2 + \Delta v)}{g \Delta u},
\]
where $ \Vert \nabla v\Vert :=  \Vert \nabla v\Vert_{C^0(\bar \Omega)}$.

 Choosing $b$ small enough (depending on $r$) we get at the point $z_0$,
 \[
 0\geq \frac{b}{2}\sum_{p=1}^n u^{p\bar p} -\frac{A_1 (1 + \Vert \nabla v\Vert^2 + \Delta v)}{g \Delta u}.  
 \]
 If $\Delta u(z_0) \leq e^{-br^2-1} (1 + \Vert \nabla v\Vert^2  +\Delta v(z_0))$, then 
 $$
 H(z_0)\leq \log (1 + \Vert \nabla v\Vert^2  +\Delta v(z_0)) - 1,
  $$ 
 hence for any $z \in \Omega$,
 $$
 \log \Delta u(z) \leq H(z) \leq H(z_0)\leq \log (1 + \Vert \nabla v\Vert^2 +\sup_{\Omega}\Delta v) - 1, 
 $$
 and the desired estimate follows. 
 
Assume now that $\Delta u(z_0) \geq e^{-br^2-1} (1+ \Vert \nabla v\Vert^2 + \Delta v(z_0))$. We then get at $z_0$ the inequality 
\begin{equation} \label{eq:Inverse}
\sum_{p=1}^n u^{p\bar p} \leq A_1 g^{-1}.
\end{equation}
Applying the inequality \eqref{eq:ElementaryIneq} with $\lambda_j := u_{j \bar j}$, we deduce from \eqref{eq:Inverse} that 
$$
\Delta u (z_0) \leq A_1 g^n g^{-(n-1)} \leq  A_2,
$$
where $A_2 := A_1 \sup_{\Omega \times [-C_0,0]} \psi^{1 \slash n}$.
Hence for any $z \in \Omega$, we have 
$$
\Delta u (z) \leq H (z) \leq H(z_0) \leq \Delta(z_0) + b \vert z_0\vert^2 \leq A_2 + b r^2.
$$
This completes the proof. 
\end{proof}

\subsection{Boundary second order a priori estimates}

\begin{pro} \label{prop:BoundaryEstimate}
Assume that  $\underline u \in PSH (\Omega) \cap C^1(\bar \Omega)$ is a strictly plurisubharmonic function on $\bar \Omega$ and   $dd^c \underline  u \geq \varepsilon \omega $ on $\Omega$, for some constant   $\varepsilon > 0$. Assume that $w\in SH(\Omega)\cap C^1(\bar \Omega)$ such that $\underline u  \leq w  < 0$ in $\Omega$  and $\underline u= w = 0$ in $\partial \Omega.$ Let $v \in PSH (\Omega) \cap C^2(\bar \Omega)$ such that $v \leq 0$ in $\bar \Omega$.  Then  any  solution $u\in PSH(\Omega) \cap C^4(\Omega) \cap C^{3}  (\bar \Omega)$   of \eqref{eq:MA} such that $\underline u \leq u \leq w$ in $\Omega$ satisfies the following a priori estimate
 
 \begin{equation}\label{eq:BoudaryEstimate}
\sup_{\partial \Omega}  \vert \nabla  u \vert + \sup_{\partial \Omega}  \Vert D^2  u \Vert \leq C_3,
 \end{equation}
where $C_3 > 0$ is a constant depending on an upper bound $C_0 := \max_{\bar \Omega} \vert \underline u \vert $,
$\Vert \underline  u\Vert_{C^2(\bar \Omega)}$, 
 $\Vert  v \Vert_{C^1 (\Omega)}$ and $\Vert \psi^{1 \slash n}\Vert_{C^{1} (\Omega \times [-C_0,0])}$, $\varepsilon^{-1}$. 
\end{pro}

\begin{proof}	
Since $\underline  u \leq u \leq w,$ we have
\begin{equation}\label{eq301}
	\sup_{\bar \Omega}\vert u \vert \leq C_0,
\end{equation}
and 
\begin{equation}\label{eq302}
	\sup_{\partial \Omega}\vert \nabla u \vert \leq C_1,
\end{equation}
where $C_0,C_1$ depending on the $C^1$-norms of $v$ and $w$ on $\partial \Omega$.

Then Proposition  \ref{prop:GradientEstimate}, yields a uniform $C^1$-estimate of $u$ in $\bar \Omega$ :
\begin{equation}\label{eq:GradientEstimate}
	\sup_{\bar \Omega}\vert \nabla u \vert \leq C'_1,
\end{equation}
where $C'_1 > 0$ is a uniform constant depending on $\Vert \underline u \Vert_{C^1(\bar \Omega)}$, 
$\Vert w \Vert_{C^{1}(\bar \Omega)}$, $\Vert \psi^{1 \slash n}\Vert_{C^{1} (\bar \Omega \times [-M,0])}$ and $ \Vert  v \Vert_{C^1 (\bar \Omega)}$.

\smallskip

The $C^2$-estimate on the boundary are done in three steps as usual.

The  estimates for tangential-tangential and tangential-normal second order derivatives follow from \cite{Guan98} (see also \cite{CKNS85}). Indeed  these estimates do not depend on a lower bound of $\psi$, and the last ones require strict plurisubharmonicity of $\underline u$  and a uniform bound on $\vert \nabla \underline u \vert$ and  $\vert \nabla u \vert$ near $\partial \Omega$. 

\smallskip
However the  estimate for the normal-normal second order derivative obtained in  \cite{Guan98} depend heavily on a positive lower bound of $\psi$ in $\bar \Omega \times \R$, while the ones obtained in \cite{CKNS85} do  require a local lower bound on $\psi$ inside $\bar \Omega \times \R$ and use   
 the  fact that $\psi $ is increasing in the second variable.

 In oder to overcome this difficulty we  use the fact that $\partial \Omega$ is strongly pseudoconvex and adapt the arguments  in \cite{Guan98} to our case.

\smallskip

We fix a point   $p \in \partial \Omega$ and let $\rho$ be a strongly plurisubharmonic defining function for $\Omega$. 
We can assume that $p=0$ is the origin and choose local coordinates $z=(z_1, \cdots, z_n)$, $z_j=x_j+i y_{j}$, $(1 \leq j \leq n)$ around the origin such that  $\rho_{z_\alpha}(0)=0$ for $\alpha<n$, $\rho_{y_n}(0)=0$, $\rho_{x_n}(0)=-1$ so that the $ x_n$-axis is the interior normal direction to $\partial \Omega$ at $0$.  For convenience  we set
$$
t_{2k-1}=x_k,~~ t_{2k}=y_k,~~ 1\leq k \leq n-1, ~~ t_{2n-1}=y_n, ~~ t_{2n}=x_n.
$$

Since $d \rho$ does not vanish near the boundary, and $u = 0$ in $\partial \Omega$, we may represent $u $ as

\begin{equation}\label{eq303}
u = h\rho,
\end{equation} 
where $h$ is a $C^2$-smooth function near the point $p=0$.

We will use the classical notations 
$$
 u_{t_\alpha} := \frac{\partial u}{\partial t_\alpha},  \, \, \, u_{t_\alpha t_\beta} = \frac{\partial^2 u}{\partial t_\alpha \partial t_\beta}, \, \, \, \, u_{j \bar k} := \frac{\partial^2 u}{\partial z_j \partial \bar z_k}\cdot
 $$

Since $\rho=0$ on $\partial \Omega$, by \eqref{eq303}, we see that $u_{x_n}(0)=-h(0)$. Observe that
near $0$ in $\Omega$ we have
$$
u_{t_\alpha t_\beta} =  h_{t_\alpha t_\beta} \rho + h_{t_\alpha} \rho_{t_\beta} + h_{t_\beta} \rho_{t_\alpha} + h \rho_{t_\alpha t_\beta} 
$$ and then
\begin{equation}\label{eq:tangential2}
	u_{t_\alpha t_\beta}(0)=-u_{x_n}(0)\rho_{t_\alpha t_\beta}(0)
	, ~~\alpha, \beta <2n.
\end{equation}
Therefore there exists a constant $C_2 > 0$ such that
\begin{equation}\label{eq304}
\vert u_{t_\alpha t_\beta}(0)\vert \leq C_2, ~~\alpha, \beta \leq 2n - 1,
\end{equation}
where $C_2 >0$ depends only on  the constant $C'_1$ in  \eqref{eq:GradientEstimate} (which depends on the $C^1$-bound of $\underline u$ at the boundary)  and an upper bound  of the Levi-form of $\rho$ on  $\partial \Omega$. 

To establish the estimates
$$
u_{t_\alpha x_n} \leq C_2, \, \, \alpha \leq 2n-1,
$$
we proceed as in  \cite{Guan98} (pages 691-693) based on the computations made in \cite[Lemma 1.3]{CKNS85}. Recall that the proof uses a generalized Hopf Lemma for the linearized operator $\mathcal L = u^{j \bar k} \frac{\partial^2}{\partial z_j \partial \bar z_k} $, which is strongly elliptic, applied to a suitable barrier function. It is easy to check that   the constant $C_2$ there depends  the $C^1$-norm of $\psi^{1\slash n} (\cdot,v)$  and a positive lower bound $\varepsilon$ of the eigenvalues of $dd^c \underline u$ (recall that $\underline  u$ is strictly plurisubharmonic on $\bar \Omega$),  but not on a lower bound of $\psi$.

 It remains to establish the estimate

\begin{equation}\label{eq305}
	\vert u_{x_n x_n}(0) \vert \leq C_3,
\end{equation}
for a uniform constant $C_3 > 0$ which does not depend on a lower bound on $\psi$.
 
We cannot rely on the proofs of \cite{CKNS85} and \cite{Guan98} where a lower bound of $\psi$ is required. However we will use an idea of \cite{CKNS85} (see p. 221) based on the fact that $\rho$ is strictly plurisubharmonic and  $- w_{x_n}(0) >0$ thanks to the Hopf Lemma.

Indeed by \eqref{eq304} we have already  the estimates 
\begin{equation}\label{eq306}
	\vert u_{p \bar q }(0) \vert  \leq C_2, ~~ p \leq n-1,  q \leq n,
\end{equation}
where $C >0$ is a uniform constant  which depends only on a $C^1$-bound of $u$ at the boundary and the geometry of $\partial \Omega$. 
Therefore it suffices to prove that

\begin{equation}\label{eq307}
	0 \leq u_{n\bar n}(0) =u_{x_n x_n}+u_{y_n y_n}(0)\leq C.
\end{equation}

Expanding $\det(u_{j\bar k})_{1 \leq j, k\leq n}$, we obtain 
\begin{equation}\label{eq308}
	\det(u_{j\bar k}(0))=au_{n\bar n}(0)+b,
\end{equation}
where 
$$
a :=\det(u_{p \bar q}(0))\vert_{\{1\leq p,q \leq n-1\}}
$$
and $b$ is bounded in view of \eqref{eq306}. 

Since $\det(u_{j\bar k})$ is bounded above, we only have to derive an a priori positive lower bound for $a$, which is equivalent to the inequality
\begin{equation}\label{eq309}
	\{u_{\alpha \bar \beta}(0)\}_{\alpha, \beta<n}\geq c_0I_{n-1},
\end{equation}
with a uniform constant $c_0>0$ which does not depend on a lower bound of $\psi$.
Indeed, from \eqref{eq:tangential2}, we deduce that
\begin{equation}\label{eq310}
u_{\alpha \bar \beta}(0)=-u_{x_n}(0)\rho_{\alpha \bar \beta}(0), ~~\alpha, \beta \leq n -1.
\end{equation}

Since $u\leq w < 0$ in $\Omega$ and $w = 0$ in $\partial \Omega$, we have by Hopf Lemma (see \cite[Lemma 6.2]{Eva10})
\begin{equation}\label{eq311}
	u_{x_n}(0)\leq w_{x_n}(0)=: -\delta_0<0.
\end{equation} 

Since $\partial \Omega$ is strongly pseudoconvex, together with \eqref{eq311} and \eqref{eq310}  we obtain the lower bound 
\eqref{eq309} with $c_0 := \delta_0$. The conclusion follows.
\end{proof}

\smallskip

\subsubsection{Global a priori estimates}
Here we deduce form the previous analysis  global a priori estimates for solutions  to the equation
\begin{equation} \label{eq:FixedPoint}
(dd^c u)^n = \psi (\cdot,u) \omega^n.
\end{equation} 

 \begin{proposition} \label{prop:GlobalEstimate} Let $\psi : \bar \Omega \times ]-\infty, 0] \longrightarrow \R^+$ be a smooth positive function such that $\psi^{1 \slash n} \in C^{1, 1} (\bar \Omega \times \R)$.  Let   $\underline  u \in C^2(\bar \Omega)$ be  a strictly plurisubharmonic function and  $w\in SH(\Omega)\cap C^1(\bar \Omega)$ such that $\underline  u \leq w < 0$ in $\Omega$  and $\underline  u=w$ in $\partial \Omega.$ 
 
 Then  if $u\in PSH(\Omega) \cap C^4(\Omega) \cap C^3 (\bar \Omega)$  is a solution to  the equation 
 \eqref{eq:FixedPoint}  such that  $\underline  u \leq u \leq w$ in $\Omega$, we have the following a priori estimates.
 
 \begin{equation}\label{eq:C1}
 \Vert \nabla u \Vert_{C^{0}(\bar \Omega)}\leq C_1,
 \end{equation}
where $C_1 > 0$ is a constant depending only on $\Vert \underline  u \Vert_{C^2(\bar \Omega)}, \Vert w \Vert_{C^{1}(\bar \Omega)}$ and $\Vert \psi^{1 \slash n}\Vert_{C^{1} (\bar \Omega \times [-M_0,0])}$ ($M_0 := \Vert u \Vert_{C^0 (\bar \Omega)}$),

and 

 \begin{equation}\label{eq:C2}
 \Vert \Delta u \Vert_{L^{\infty}(\Omega)}\leq C_2,
 \end{equation}
where $C_2 > 0$ is a constant depending only on $\Vert \underline  u \Vert_{C^2(\bar \Omega)}, \Vert w \Vert_{C^{1}(\bar \Omega)}$ and $\Vert \psi^{1 \slash n}\Vert_{C^{1,1} (\bar \Omega \times [-M_0,0])}$.
 \end{proposition}
 
  \smallskip
 
\begin{proof}	
Since $\underline  u \leq u \leq w \leq 0$ in $\bar \Omega$,  we have
\begin{equation}\label{eq301}
	\sup_{\bar \Omega}\vert u \vert \leq C_0 := \sup_{\bar \Omega}\vert \underline u \vert,
\end{equation}
and 
\begin{equation}\label{eq302}
	\sup_{\partial \Omega}\vert \nabla u \vert \leq C_1 := \max \{\sup_{\partial \Omega}\vert \nabla \underline u \vert, \Vert \nabla w \Vert_{L^{\infty}(\partial \Omega)}\}\cdot
\end{equation}

Then Corollary  \ref{coro:GradientEstimate2} yields a uniform $C^1$-estimate of $u$ in $\bar \Omega$ :
\begin{equation}\label{eq3}
	\sup_{\bar \Omega}\vert \nabla u \vert \leq C'_1,
\end{equation}
where $C'_1 > 0$ is a uniform constant.

\smallskip

Then we can apply Proposition \ref{prop:BoundaryEstimate} to obtain a uniform estimate of $\Delta u$ on $\partial \Omega$. 
To conclude it's enough to apply Proposition \ref{prop:LaplaceEstimate}.
\end{proof}
\smallskip
\subsection{Proof of Theorem \ref{thm:new}}
Recall that we want to solve the following Dirichlet problem.
\begin{equation} \label{eq:DP0}
\left\lbrace
\begin{array}{lcr}
(dd^c u)^n  = \psi (\cdot,u) \, \omega^n & \textnormal{on} & \Omega\\
u=0& \textnormal{ on} & \partial \Omega.
\end{array}
\right.
 \end{equation}
 
This is a fixed point problem, so we will use an iterative method to solve it. Since $\psi$ can degenerate at the boundary, we  proceed in two steps.

\smallskip
{\bf Step 1.} We assume that $\psi > 0$ in $\bar \Omega \times ]-\infty,0]$. 
  Applying Theorem \ref{thm:CKNS}, we will  construct inductively a sequence $(u_j)$ in $PSH (\Omega) \cap C^{\infty}(\bar \Omega)$ such that   $u_0 := \underline u$ and for each  $j \in \N$,  $u_{j+1} = T(u_j)$ is the unique solution to  the following problem
\begin{equation} \label{eq:DP1}
\left\lbrace
\begin{array}{lcr}
(dd^c u_{j+1})^n  = \psi (\cdot,u_j) \, \omega^n & \textnormal{on} & \Omega\\
u_{j+1} =0& \textnormal{ on} & \partial \Omega.
\end{array}
\right.
 \end{equation}
 
 Now the main observation is that the sequence $(u_j)_{j \in \N}$ is increasing and satisfies $\underline u \leq u_j \leq u_{j+1} \leq \bar u$ in $\Omega$. This follows from the the comparison principle Proposition \ref{prop:CP} and the fact that $\psi$ is non increasing.
 
 Indeed for $j = 0$ we set  $u_0 := \underline u$. 
 Assume that for a fixed $j \in \N^*$, we have constructed functions $ (u_k)_{1 \leq k \leq j}$ in $ PSH(\Omega) \cap C^{\infty}({ \bar \Omega })  $  such that for $1 \leq k \leq j$, we have
 $$ 
 (dd^c u_{k})^n  = \psi (\cdot,u_{k -1})\,  \omega^n, \, \, \text{on} \, \, \Omega \, \, \, \text{and} \, \, \,  u_k = 0 \, \, \text{in} \, \, \,  \partial \Omega,
 $$ 
 on $\Omega$ and  $ \underline  u  \leq u_{k-1} \leq  u_{k} \leq \bar u$ on $\Omega$.
 
 By Theorem \ref{thm:CKNS}, the Dirichlet problem \eqref{eq:DP1} admits a unique solution
 $u_{j+1} \in  PSH(\Omega) \cap C^{\infty}({ \bar \Omega })$.

 Since   $\underline u \leq u_{j-1} \leq u_j  \leq \bar u$ in $\Omega$ and $\psi$ is non increasing in the last variable, we have
 \begin{eqnarray*}
 (dd^c \bar u)^n \leq \psi(\cdot, \bar u) \, \omega^n   &\leq & \psi (\cdot,u_j) \,   \omega^n = (dd^c u_{j+1})^n \\
 & \leq &  \psi (\cdot,u_{j-1})\,  \omega^n  = (dd^c u_{j})^n \\
 &  \leq & \psi (\cdot,\underline u) \, \omega^n  \leq (dd^c \underline u)^n, \, \, \text{on} \, \, \Omega. 
 \end{eqnarray*}
 It follows from the comparison principe  that  $\bar u \geq u_{j+1} \geq u_j \geq \underline u$ in $\Omega$.
 This achieves the construction.

\smallskip

Since the sequence  $(u_j)$ is increasing, it converges a.e. to a function $u \in PSH (\Omega) \cap L^{\infty} (\Omega)$ such that $\underline u \leq u \leq \bar u$ in $\Omega$. 

Setting $K_j:= \sup_{\bar{\Omega}}|\nabla u_j|^2$, we see  by Proposition \ref{prop:GradientEstimate} and Proposition \ref{prop:BoundaryEstimate} that
	\[
	K_{j+1}\leq C_1+ 2^{-1}K_j,
	\]
 for any $j$ with  a uniform constant $C_1$ independent of $j$. Thus $K_j\leq 2C_1 + K_0$, for all $j$, hence the sequence  $(\Vert \nabla u_j\Vert_{C^0(\bar \Omega)}$ is  bounded by a constant which does not depend on $j$. It follows from Ascoli-Arzel\`a theorem that actually $u_j \to u$ uniformly on $\bar \Omega$ and $u \in PSH (\Omega) \cap C^{0,1} (\bar \Omega)$. 

Moreover, passing to the limit in the weak sense in \eqref{eq:DP1}, we see that $u $ is a weak solution to the Dirichlet problem \eqref{eq:DP0}.

On the other hand, the inequalities  $\underline u \leq u_j \leq \bar u$ on $\bar \Omega$, yields 
$$
\sup_{j} \left(\Vert u_j \Vert_{C^0(\bar \Omega)} + \Vert \nabla u_j\Vert_{C^0(\partial \Omega)}\right)< + \infty
$$  
By Proposition \ref{prop:GradientEstimate}, it follows  that  $\sup_j \Vert \nabla u_j\Vert_{C^0(\bar \Omega)}  < + \infty$.
 It follows from   the boundary Laplacian estimate Proposition \ref{prop:BoundaryEstimate} applied to $u  = u_{j + 1}$ and $v= u_j$ (here we only need the bound for the gradient $\nabla u_j$) that the sequence
 $(\Vert \Delta u_{j+1} \Vert_{ C^0(\partial \Omega)})$ is   bounded.  
 Hence there exists a uniform constant  $C' > 0$ such that for any $j \in \N$, 
 $$
  \Vert \nabla u_{j }\Vert_{C^0(\bar \Omega)} +  \Vert \Delta u_{j+1} \Vert_{ C^0(\partial \Omega)} \leq C'.
 $$
 Applying Proposition \ref{prop:LaplaceEstimate} with $u = u_{j+1}$ and $v = u_j$, we see that the constants $L_j :=  \Vert \Delta u_j \Vert_{L^{\infty}(\Omega)}$ satisfy the inequality 
 $L_{j+1} \leq C_2 + 2^{-1}  L_j$ for any $j \in \N$. Hence $L_j \leq 2 C_2 + L_0$ for any $j \in \N$.

Therefore there exists a uniform constant $C'' > 0$  such that
 \begin{equation} \label{eq:Lestimate}
 \Vert u_{j}\Vert_{C^{1, \bar 1}(\bar \Omega)} \leq C''.
 \end{equation}
  for any $j \in \N$, where $C''$ depends only on $C'$, the uniform barriers $ \underline u$ and $ \bar u$ of the $u_j's$ and an upper bound of $\psi$ on $\bar \Omega \times [-M_0, 0]$ as well as a lower bound of $\psi$ on $\Omega_0 \times [-M_0, - A_0]$, the later depending only on $\Omega$ and $\bar u$ and an upper bound  of $\psi$ on $\bar \Omega \times [-M_0,0] $ where $M_0$ is a uniform  bound of the $u_j's$ (see Proposition \ref{prop:GradientEstimate}).
  
  Taking the limit in \eqref{eq:Lestimate} we obtain $\Vert u\Vert_{C^{1, \bar 1}(\bar \Omega)} \leq C.$ Hence $u \in C^{1, \bar 1}(\bar \Omega)$.

\smallskip

{\bf Step 2 :} Assume now that $\psi \geq 0$, $\psi > 0$ in $\Omega \times ]-\infty,0[$ and $\underline u$ is a strict subsolution to the Dirichlet problem \eqref{eq:MA} in the sense that it satisfies the inequality \eqref{eq:StrictSubsolution} for some constant $\epsilon_0$. For  a fixed  $0 < \varepsilon \leq 1$, set $\psi_\varepsilon := \psi + \varepsilon^n$. Then consider  the Dirichlet problem
\begin{equation} \label{eq:DPepsilon}
\left\lbrace
\begin{array}{lcr}
(dd^c u)^n= \psi_ \varepsilon (\cdot,u) \,  \omega^n & \textnormal{in} & \Omega\\
 u =\phi& \textnormal{ on} & \partial \Omega.
\end{array}
\right.
 \end{equation}
Since $  \psi_ \varepsilon  \geq  \varepsilon^n > 0 $,  we want to apply the first step to  solve this problem. To this end we need to construct a subsolution $ \underline u_{\varepsilon}$ and a supersolution $\bar u_{\varepsilon}$ to the problem \eqref{eq:DPepsilon} such that $\underline u_{\varepsilon} \leq \bar u_{\varepsilon}$ in $\Omega$. 

Indeed, observe that $\bar u_{\varepsilon} := \bar u$ is clearly a supersolution since $\psi \leq \psi_\varepsilon$. By strong pseudoconvexity, the domain $\Omega$ admits a  smooth strictly purisubharmonic defining function $\rho$ such that $dd^c \rho \geq \omega$ in $\Omega$. For  $0 < \varepsilon < 1$, the function $\underline  u_\varepsilon := \underline  u + \varepsilon \rho$ is plurisubharmonic on $\Omega$ and $\underline  u + \rho \leq \underline  u_\varepsilon \leq \underline  u$ in $\Omega$.  Moreover if we set
$$
M := \Vert \rho\Vert_{C^0(\bar \Omega)} \cdot \Vert \partial_t \psi \Vert_{C^0(\bar \Omega \times I)},
$$ 
where $I := [- C_0,0]$ and $C_0 := \Vert \underline  u + \rho\Vert_{C^0(\bar \Omega)}$,
we get 
\begin{eqnarray*}
(dd^c \underline  u_\varepsilon )^n &\geq& (dd^c \underline  u)^n + \varepsilon^n (dd^c  \rho)^{n} 
\geq (\psi (\cdot,\underline  u) + \epsilon_0 + \varepsilon^n) \omega^n \\
&\geq & (\psi (\cdot,\underline  u_\varepsilon) - M \varepsilon +  \epsilon_0 + \varepsilon^n) \omega^n \geq  (\psi (\cdot,\underline  u_\varepsilon) + \varepsilon^n) \omega^n,
\end{eqnarray*}
if  we choose $\varepsilon > 0$ so small that $M \varepsilon \leq \epsilon_0$.

Hence for $0 < \varepsilon < \varepsilon_0 := \min \{1,\epsilon_0 \slash M\}$, the function $\underline  u_\varepsilon$ is
a subsolution to the Dirichlet problem \eqref{eq:DPepsilon} such that $\underline  u_\varepsilon \leq  \underline u \leq  \bar u$ in $\bar \Omega$.

From the first case, it follows that for any $0 < \varepsilon < \varepsilon_0 $, the Dirichlet problem \eqref{eq:DPepsilon} admits a solution $ u_{\varepsilon} \in PSH (\Omega) \cap C^{1, \bar 1}(\bar \Omega)$ such that $\underline  u + \rho \leq  \underline  u + \varepsilon \rho \leq u_\varepsilon \leq \bar u $ in $\Omega$. 

Using Ascoli-Arzela theorem we can extract a subsequence $(u_{\varepsilon_j})$ which converges uniformly on $\bar \Omega$ to a function $u \in PSH (\Omega) \cap C^{0,1}(\bar \Omega) $. 
Hence $(dd^c u)^n = \psi (\cdot,u) \omega^n$ in the weak sense on $\Omega$ and $u= 0$ on $\partial \Omega$ i.e. $u$ is a weak solution to the Dirichlet problem \eqref{eq:DP0}.

Observe that we have for any $\varepsilon \in ]0,\varepsilon_0[$, $\underline u_1 := \underline u + \rho \leq u_\varepsilon \leq \bar u < 0$ on $\Omega$. 

We claim that  there exists a uniform constant $C > 0$  independent of $\varepsilon \in ]0,\varepsilon_0[$ such that
\begin{equation} \label{eq:Fest2}
\Vert  u_\varepsilon \Vert_{C^{1,\bar 1} (\bar \Omega)}  \leq C.
\end{equation}

Indeed fix $0 < \varepsilon < \varepsilon_0 $.  We know   from the first step that by construction  $u_\varepsilon$ is the limit of a sequence $(u_\varepsilon^j)_{j \in \N}$ in $PSH(\Omega) \cap C^{\infty}(\bar \Omega)$ satisfying the equation
$$
(dd^c u_\varepsilon^{j +1})^n =  (\psi (\cdot,u_\varepsilon^j) + \varepsilon^n) \omega^n,
$$
on $\Omega $ and  $\underline u_\varepsilon \leq u_\varepsilon^j \leq \bar u$ on $\Omega$ for any $ j \in \N $.

Since $\underline u_1 := \underline u + \rho \leq  u_\varepsilon^j \leq \bar u$  on $\Omega$ for any $\varepsilon \in ]0,\varepsilon_0[$ and any $j \in \N$, 
it follows from  the estimates \eqref{eq:Lestimate} in the first step that 

 \begin{equation} \label{eq:Final2}
 \Vert u_\varepsilon^j\Vert_{C^{1, \bar 1}(\bar \Omega)} \leq C,
 \end{equation}
where $C > 0$ does not depend on $\varepsilon, j$. 
Passing to the limit in \eqref{eq:Final2} as $j \to + \infty$ for fixed $\varepsilon \in ]0,\varepsilon_0[$, we obtain the required estimate \eqref{eq:Fest2}, which proves our claim.

By Ascoli-Arzela theorem there exists a sequence $\varepsilon_j \to 0$ such that $u_j := u_{\varepsilon_j} \to u$ uniformly on $\bar \Omega$, where $u \in PSH(\Omega) \cap C^{0,1} (\bar \Omega)$ and $\underline u \leq u \leq  \bar u < 0$ on $\Omega$.

Moreover passing to the limit in \eqref{eq:Fest2} as $\varepsilon = \varepsilon_j \to 0$, we obtain the required estimate
$$
\Vert  u \Vert_{C^{1,\bar 1} (\bar \Omega)}  \leq C,
$$
which  proves that $u \in C^{1,\bar 1} (\bar \Omega)$.

 To prove that $u \in C^{\infty}(\Omega)$, we will use the complex analogue of Evans-Krylov type argument.

 Recall that for each $\varepsilon \in ]0,\varepsilon_0[$,  $u_\varepsilon \in PSH(\Omega) C^{1,\bar 1} (\bar \Omega)$ is a solution to the following Monge-Amp\`ere equation 
 $$
 (dd^c u_\varepsilon)^n = g_\varepsilon  \, \omega^n \, \, \text{on} \, \, \Omega, \, \, \, \text{where} \, \, \, g_\varepsilon  (z):=  \psi(z,u_\varepsilon(z))   + \varepsilon^n.
 $$
 Observe that the right hand side $g_\varepsilon $ of this equation  is locally uniformly bounded from below by a positive constant independent of $\varepsilon$. 
 
 Indeed we have  $g_\varepsilon (z)  \geq  \psi (z,u_\varepsilon(z)) \geq \psi (z,\bar u(z)) > 0 $ in $\Omega$ since $\psi$ is non increasing and $ u_\varepsilon \leq \bar u < 0$ in $\Omega$. 
 Moreover $g_\varepsilon^{1 \slash n}  \in C^{0,1}(\bar \Omega)$ and $\Vert g_\varepsilon^{1 \slash n} \Vert_{C^{0,1}(\bar \Omega)}$ is uniformly bounded by a constant independent of $\varepsilon \in ]0,\varepsilon_0[$.
 
 We can then apply the complex Evans-Krylov local argument  provided in \cite{Wang12} (see also \cite{DZZ11} and \cite[Theorem 14.9]{GZ17}). 
 Then  for  any subdomain $\Omega' \Subset \Omega$ and any  $\alpha \in ]0,1[$, there exists 
 a constant $C' >0$  such that  for any  $\varepsilon \in ]0,\varepsilon_0[$,
$$
 \Vert u_\varepsilon \Vert_{C^{2,\alpha}(\Omega')} \leq C',
  $$
 where $C'$ does not depend on $\varepsilon \in ]0,\varepsilon_0[$.
  It follows that $u \in C^{2,\alpha}(\Omega')$. 
 
 To conclude that $u \in C^{\infty}(\Omega)$, we will use the Schauder theory as explained in  \cite[Section 14.3.2]{GZ17}. 
 More precisely the operator 
 $$
 F (u) := \text{det} (u_{j \bar k}) 
 $$
 is locally uniformly elliptic  on the class of smooth strictly plurisubharmonic functions $u $ on $\Omega$.
 
 Fix an open set $\Omega'' \Subset \Omega$ and take a subdomain $\Omega'$ such that $\Omega'' \Subset \Omega '\Subset \Omega$.
 Recall that $u   \in C^{2,\alpha}(\Omega')$ is a solution to the following equation on $\Omega'$ 
  $$
 F (u) = g := \psi(\cdot,u).
 $$
 Since $g \in C^{\alpha} (\Omega')$, for a fixed $\zeta \in \C^n$ with $\vert \zeta \vert = 1$, the difference quotients 
 $$
 u^h (z) := \frac{u (z +h \zeta) - u(z)}{h}, 
 $$
 are defined in a neighborhood $V$ of $\bar \Omega''$ for $\vert h\vert > 0 $ small enough, and satisfies a uniformly elliptic second order linear equation
 $L^h u^h = g^h$ in $V$ with H\"older continuous coefficients and H\"older continuous right hand side (see \cite[Section 14.3.2]{GZ17}, \cite[Corollary 6.3]{GT98}).
  By the interior Schauder a priori estimates in \cite[Corollary 6.3]{GT98}, it follows that $u^h \in C^{2,\alpha}(V)$ with a uniform control on its $C^{2,\alpha}(\Omega'')$-norm (independent of $h$ and $\zeta$). Hence $u\in C^{3,\alpha} (\Omega'')$ for any $\Omega'' \Subset \Omega$.
By iterating this argument we conclude that  $u \in C^{k}(\Omega'')$ for any $k \in \mathbb N$ and any $\Omega'' \Subset \Omega$. Hence $u \in C^{\infty}(\Omega)$.


\section{The existence of the first eigenvalue}

We will now give the proof of  Theorem \ref{thm1}, deduce some properties of the first eigenvalue and give an application.

\subsection{Proof of Theorem \ref{thm1}}
We follow the strategy of Lions \cite{Lions86}.

Recall that for $a\in \mathcal{A}(\Omega)$, the first eigenvalues of the operator $- L_a$ is given by te formula 

\begin{equation}\label{eq317}
\gamma_1(a) := \sup \{\gamma \geq 0:~ \exists u \in C^2(\Omega),~~ u <0 ~~\textnormal{and}~~ L_a u \geq -\gamma u f \}.
\end{equation} 

\begin{proof}
We define the following number
\begin{equation}\label{eq318}
	\lambda_1:=\inf\{ \gamma_1(a):~~a\in \mathcal{A}(\Omega)\}.
\end{equation}

 Following Lions, we consider the Dirichlet problem
	
	\begin{equation}\label{eq319}
	\left\lbrace
	\begin{array}{lcr}
	(dd^c  u)^n=(1-\lambda u)^nf^n\omega^n & \textnormal{in} & \Omega,\\
	u\equiv 0& \textnormal{ on} & \partial \Omega,
	\end{array}
	\right.
	\end{equation}
	where $\lambda > 0$ is a parameter.
	
	\smallskip
	
	Then we introduce the  following real number :
	
	\begin{equation}\label{eq320}
		\mu_1:=\sup\{ \lambda \geq 0:~ \exists u\in \mathcal P_0 (\Omega),~ \textnormal{solution of} ~\eqref{eq319} \},
	\end{equation}
	where 
\begin{equation}\label{eq316}
	\mathcal{P}_0(\Omega):=\{ u \in PSH(\Omega) \cap C^{2}(\Omega) \cap C^0(\bar \Omega) \, ; \,  u_{|_{\partial \Omega}}=0\}.
\end{equation}
	
 By  Theorem \ref{thm:CKNS} there exists $u_0 \in PSH (\Omega) \cap C^{\infty} (\bar \Omega)$ solution to the  
 problem \eqref{eq319} with $\lambda = 0$ i.e. $(dd^c u_0)^n = f^n \omega^n$.  Hence $\mu_1$ is well defined and $0\leq \mu_1 \leq +\infty$.
	 	 
  Observe that since the right hand side $\psi (z,u) := (1-\lambda u)^n f(z)^n \geq f^n (z) $  in $\bar \Omega \times ]-\infty , 0]$, the function $u_0$ is a supersolution to the Dirichlet problem \eqref{eq319} for any $\lambda > 0$.	
 \smallskip
	 
 The proof will be done in several steps.

	 \smallskip
	 
 {\bf Step 1 :} We  show that $ \mu_1 \leq \lambda_1.$
 
 Indeed if  $\lambda \geq 0$ is such that $u_\lambda$ is a solution to the problem \eqref{eq319},  by the formula \eqref{eq217}
 for any $a \in \mathcal A(\Omega)$ we have
 $$
 L_a u \geq (1-\lambda u_\lambda)f ~\textnormal{in}~ \Omega, ~ u_\lambda<0 ~\textnormal{in}~ \Omega, ~u_\lambda=0~ \textnormal{on}~\partial \Omega.
 $$
	
 Then $\lambda \leq \gamma_1 (a) $. Therefore he have $0 \leq \mu_1 \leq \lambda_1 < \infty$. 

\smallskip 
	
{\bf Step 2 :} We  prove that $\mu_1 \geq \Vert u_0 \Vert_{C^0(\bar \Omega)}^{-1}$.

 Indeed, if $\lambda <\Vert u_0 \Vert_{C^0(\bar \Omega)}^{-1}$, we set $C=(1-\lambda \Vert u_0 \Vert_{C^0(\bar \Omega)})^{-1}$. Then $\underline u=Cu_0 \in \mathcal{P}_0(\Omega)$ and
	$$
	(dd^c\underline u)^n=(dd^cCu_0)^n=C^n(dd^cu_0)^n=C^nf^n\omega^n.
	$$
 Also
 \begin{align*}
 Cf-(1-\lambda\underline u)f&=\lambda Cf(\Vert u_0 \Vert_{C^0(\bar \Omega)}+u_0)\\&\geq 0.
 \end{align*}
 Hence 
 \begin{equation}\label{eq321}
 \left\lbrace
 \begin{array}{lcr}
  (dd^c \underline u)^n\geq (1-\lambda \underline u)^nf^n\omega^n & \textnormal{on} & \Omega,\\
 \underline u\equiv 0& \textnormal{ on} & \partial \Omega.
 \end{array} \right.
\end{equation}
 In other words $\underline u$ is a subsolution of the problem \eqref{eq319}.	
 Since   $(1-\lambda \underline u)^n f^n \omega^n \geq f^n \omega^n = (dd^c u_0)^n$ in $ \Omega$, it follows	
 that $u_0$ is a supersolution of the problem \eqref{eq319} such that $\underline u \leq u_0$ in $\Omega$. Since the right hand side of the equation \eqref{eq319} is positive, we can apply  Theorem \ref{thm:new} to conclude that for $0<\lambda<\Vert u_0 \Vert_{C^0(\bar \Omega)}^{-1}$  there exists $u_\lambda \in \mathcal P_0 (\Omega) \cap C^{\infty} (\Omega)$ solution of \eqref{eq319}  such that $\underline u_1 \leq u_\lambda \leq u_0$ in $\Omega$. This proves that any   $\lambda<\Vert u_0 \Vert_{C^0(\bar \Omega)}^{-1}$, satisfies the inequality  $\lambda \leq \mu_1$. Therefore we have shown that 
  \begin{equation} \label{eq:minoration}
  \Vert u_0\Vert_{C^0(\bar \Omega)}^{-1}\leq \mu_1 \leq \lambda_1.
 \end{equation}	
  
 \smallskip
 
{\bf Step 3 :} We prove that  for any $0<\lambda<\mu_1$, the problem \eqref{eq319} admits a solution $u_\lambda$.

Indeed, if $0<\lambda<\mu_1$,  there exists $\lambda< \lambda'<\mu_1$ and $u_{\lambda'}\in \mathcal P_0(\Omega)$  
 such that
 $$
 (dd^cu_{\lambda'})^n=(1-\lambda'u_{\lambda'})^nf^n\omega^n\geq (1-\lambda u_{\lambda'})^nf^n\omega^n.
 $$
 Therefore $u_{\lambda'}$ is a subsolution to the problem \eqref{eq319}. Since we already know that $u_0$ is a supersolution, applying again Theorem \ref{thm:new}  yields the existence of  a solution $u_\lambda \in \mathcal P_0(\Omega)$.
	
 \smallskip
	
 {\bf Step 4  :}  We show that $\Vert u_\lambda\Vert_{C^0(\bar \Omega)} \to  + \infty$ as $\lambda \to \mu_1.$
	
  Assume the converse is true i.e.  there exists $M>0$ such that for some subsequence $(\lambda_j)$ in $]0, \mu_1[$, converging to $\mu_1$ we have, $\Vert u_{\lambda_j} \Vert_{C^0(\bar \Omega)}\leq M<+\infty.$ Then	
  $$
  -M\leq u_{\lambda_j}\leq 0 ~~\textnormal{in}~ \bar \Omega.
  $$
 and for any $ \lambda = \lambda_j$, we have 
 $$
 (dd^cu_\lambda)^n=(1-\lambda u_\lambda)^nf^n\omega^n\leq (1+\mu_1M)^nf^n\omega^n.
 $$
	Observe that $\underline u := (1+\mu_1 M) \, u_0$ solves the following problem 
	\begin{equation}
	\left\lbrace
	\begin{array}{lcr}
	(dd^c \underline u)^n=(1+\mu_1 M)^nf^n\omega^n& \textnormal {in} &\Omega,\\
	\underline u =0 & \textnormal{on} & \partial\Omega,
	\end{array}
	\right.
	\end{equation} 
	 and then by the comparison principle  $\underline u \leq u_\lambda$ on $\bar\Omega$.
	
	On the other hand, since 
	$$
	(dd^c u_\lambda)=(1-\lambda u_\lambda)^nf^n\omega^n\geq f^n\omega^n,
	$$
 and $w := u_0$ solves the problem
 \begin{equation}
 \left\lbrace
 \begin{array}{lcr}
 (dd^c w)^n=f^n\omega^n& \textnormal {in} &\Omega,\\
 w=0 & \textnormal{on} & \partial\Omega,
 \end{array}
 \right.
 \end{equation} 
 and we have $ u_\lambda \leq w$ on $\bar \Omega$.
	
 Thus $\underline u \leq u_\lambda \leq w < 0$  on $ \Omega$. Since  $\underline u$ is strictly plurisubharmonic on $\Omega$, it follows  from Proposition \ref{prop:GlobalEstimate}, that there exists a uniform 
 constant  $C>0$  such that for any $\lambda = \lambda_j$
 $$ 
 \Vert u_\lambda\Vert_{C^{1,\bar 1}(\bar{\Omega})}\leq C.
 $$ 	
 By the same reasoning as in the proof of Theorem \ref{thm:new} (Step 2) based on the complex analogue of Evans-Krylov theory we conclude that, up to extracting a subsequence we can assume that  $\{u_{\lambda_j}\}$  converges in the $C^{2}(\bar \Omega')$-norm for any $\Omega' \Subset \Omega$ to a function $u_{\mu_1} \in \mathcal P_0(\Omega)$ solution of \eqref{eq319} \ when $\lambda \to \mu_1.$ 
 
 As in the second step,  we see that for $\delta<\Vert u_{\mu_1}\Vert_{C^0(\bar \Omega)}^{-1}$, $C :=(1-\delta \Vert u_{\mu_1}\Vert_{C^0(\bar \Omega)})^{-1}$ and $\underline u := Cu_{\mu_1} \in  \mathcal P_0 (\Omega)$,  and we have
	
 $$	
 (dd^c\underline u)^n=C^n(dd^c u_{\mu_1})^n=C^n(1-\mu_1u_{\mu_1})^nf^n\omega^n,
 $$
 and
 \begin{equation*}
 C(1-\mu_1u_{\mu_1})-(1-(\delta + \mu_1)Cu_{\mu_1})=C\delta(\Vert u_{\mu_1}\Vert_{C^0(\bar \Omega)}+u_{\mu_1})\geq 0.
 \end{equation*}
	Therefore
	
	\begin{equation}
	\left\lbrace
	\begin{array}{lcr}
	(dd^c \underline u)^n\geq (1-(\delta+ \mu_1)\underline u )^nf^n\omega^n& \textnormal {in} &\Omega,\\
	\underline u=0 & \textnormal{on} & \partial\Omega,
	\end{array}
	\right.
	\end{equation} 
 In other words $\underline u$ is a subsolution to the problem \eqref{eq319}. Again   since $u_0$ is a supersolution,  
 it follows from Theorem \ref{thm:new}  that there exist a solution  $u_\lambda$  of \eqref{eq319} for $\mu_1 \leq 
 \lambda <\mu_1+ \delta$ and this contradicts the definition of $\mu_1.$ 
 Hence , $\Vert u_{\lambda}\Vert_{C^0(\bar \Omega)} \to \infty$ as $\lambda \to \mu_1.$
 
 \smallskip
	
 {\bf Step 5  :}  We show that  the normalized family defined  for $\lambda<\mu_1$ by the following formula :
	$$
	v_\lambda : = \frac{u_\lambda}{ \Vert u_{\lambda}\Vert_{C^0(\bar \Omega)}},
	$$ 	
 converges to a function $\f_1\in PSH(\Omega)\cap C^{1,\bar 1}(\bar \Omega) \cap C^2(\Omega)$ such that $(\mu_1,\varphi_1)$ is a solution to the eigenvalue problem \eqref{eq1}. 
 
 Indeed, for $\lambda <\mu_1$, we have
  $$
 (dd^cv_\lambda)^n=\Vert u_{\lambda}\Vert_{C^0(\bar \Omega)}^{-n}(1-\lambda   
  u_\lambda)^nf^n\omega^n=(\Vert u_{\lambda}\Vert_{C^0(\bar \Omega)}^{-1}-\lambda 
  v_\lambda)^nf^n\omega^n ~\textnormal{in}~ \Omega.
 $$
Hence $v_\lambda$ solves the following problem :
 \begin{equation}\label{eq325}
 \left\lbrace
 \begin{array}{lcr}
 (dd^c v_\lambda)^n=(\Vert u_{\lambda}\Vert_{C^0(\bar \Omega)}^{-1}-\lambda v_\lambda)^nf^n\omega^n&  
 \textnormal {in} &\Omega\\
 v_\lambda =0 & \textnormal{in} & \partial\Omega\\
 \Vert v_\lambda \Vert_{C^0(\bar \Omega)}  = 1.
\end{array}
	\right.
	\end{equation} 
	Using the function $u_0$ defined before, we claim that  there exists $C > 0$ and $0 < \mu_2< \mu_1$
	 such that for $\mu_2 < \lambda < \mu_1$, we have  
	$$
	C u_0 \leq v_\lambda \leq 0~ \textnormal{in}~ \bar \Omega.
	$$
 Indeed by the  fourth step, there exists $0 < \mu_2< \mu_1$ such that  for $\mu_2 < \lambda<\mu_1$ we have 
 $\Vert u_{\lambda}\Vert_{C^0(\bar \Omega)} >  1$. Since $v_\lambda<0$ in $\Omega$ and 
 $\Vert v_{\lambda}\Vert_{C^0(\bar \Omega)}=1$, it follows that  for $\mu_2 < \lambda<\mu_1$,
	$$
	(\Vert u_{\lambda}\Vert_{C^0(\bar \Omega)}^{-1}-\lambda v_\lambda)\leq 1+\mu_1.
	$$
	Choosing $C=1+\mu_1$, we have  for  $\mu_2 < \lambda<\mu_1$,
	
	\begin{equation}\label{eq326}
	\left\lbrace
	\begin{array}{lcr}
	(dd^c v_\lambda)^n\leq C^nf^n\omega^n\leq (dd^c (C u_0))^n& \textnormal {in} &\Omega\\
	v_\lambda = 0 & \textnormal{on} & \partial\Omega.
	\end{array}
	\right.
	\end{equation} 
	By comparison principle, we deduce that $C u_0 \leq v_\lambda<0$ in $\Omega$ for $\mu_2 < \lambda<\mu_1$.
	This proves the claim.
	
 Therefore we have a uniform bound on $\Vert \nabla v_\lambda\Vert_{C^0(\partial \Omega)} $ 
 independent of $\lambda  \in ]\mu_2, \mu_1[$.  
 Since $\Vert v_\lambda\Vert_{C^0(\bar \Omega)} = 1$, it follows  from Corollary    \ref{coro:GradientEstimate2}  that  
 there   exists a constant $C_1 > 0$ such that for any $\lambda \in ]\mu_2, \mu_1[$, 
 $$
 \Vert v_\lambda\Vert_{C^1(\bar \Omega)} \leq C_1
 $$
 By Arzelà-Ascoli theorem, there exists a subsequence $(\lambda_j)$ in $ ]\mu_2, \mu_1[$ such that $\lambda_j \to \mu_1$   and   
  the subsequence $(v_{\lambda_j})$ converges uniformly on $\bar \Omega$ to some $\f_1 \in  PSH (\Omega) \cap C^{0,1}(\bar \Omega)$ which is a weak solution of the following problem :
	
	\begin{equation}\label{eq327}
	\left\lbrace
	\begin{array}{lcr}
	(dd^c \f_1)^n=(-\mu_1 \f_1)^n f^n\omega^n & \textnormal{in} & \Omega\\
	\f_1=0& \textnormal{ on} & \partial \Omega\\
	\Vert \f_1 \Vert_{C^0(\bar \Omega)}=1.& &
	\end{array}
	\right.
	\end{equation}
	
 To prove that $\varphi_1 \in C^{1,\bar1}(\bar \Omega)$, we will apply Proposition \ref{prop:GlobalEstimate}. This requires the existence of a uniform 
  super-barrier for all $v_\lambda$'s. This will  be constructed using $\varphi_1$.
	
 Indeed since $v_{\lambda_j} \to \varphi_1$ uniformly in $\bar \Omega$ as $j \to + \infty$,   it follows that
	$$
	-v_{\lambda_j} \geq -\f_1-\frac{1}{3}
	$$
 pointwise in $\bar \Omega$ for any $j \in \N$ large enough.
 Since $\Vert \f_1 \Vert_{C^0(\bar \Omega)}=1$, there exists $z_0 \in \Omega$ such that $-\f_1(z_0)=1$. 
 Hence, there exists a ball $B=B(z_0,r)\Subset \Omega$ such that $-\f_1 \geq\frac{2}{3} $ in $\bar B$. 
 Therefore
 $$
 -v_{\lambda_j} \geq \frac{1}{3}
 $$
 in $\bar B$ for any $j$ large enough.  It follows that
 $$
 (dd^cv_{\lambda_j})^n\geq 3^{-n}\lambda_j^n f^n \omega^n \geq 3^{-n}\mu_2^n f^n \omega^n
 $$
 in $\bar B$ and  therefore
 $$
 \Delta v_{\lambda_j} \geq \frac{\mu_2}{3}f
 $$
 in $\bar B$.
	
 Now let us solve  the following classical Dirichlet problem for the Laplace operator
 \begin{equation}\label{eq328}
 \left\lbrace
 \begin{array}{lcr}
 \Delta w = \frac{\mu_2}{3}f\theta_B & \textnormal {in} &\Omega\\
 w=0 & \textnormal{on} & \partial\Omega,
 \end{array} \right.
 \end{equation}
 where $0\leq\theta_B\leq 1$ is a smooth function in $\bar \Omega$ with $\text{Supp} (\theta_B)\subset B$  
 and $\theta _B \equiv 1$ in a neighborhood of $z_0$. Then we obtain a smooth subharmonic function $w$  
 such that
 \begin{equation}\label{eq329}
 \left\lbrace
 \begin{array}{lcr}
 \Delta w \leq \Delta v_{\lambda_j} & \textnormal {in} &\Omega\\
 w=v_{\lambda_j} =0 & \textnormal{on} & \partial\Omega.
 \end{array} \right.
 \end{equation}
 Therefore $Cu_0\leq v_{\lambda_j}  \leq w<0$ in $\Omega$ for $j$ large enough. 
 
 By Proposition \ref{prop:GlobalEstimate},  there   
  exists a positive constant $C > 0$ such that for  $j$ large enough, we have
  
  \begin{equation} \label{eq:boundedLaplacian}
  \Vert v_{\lambda_j}  \Vert_{C^{1,\bar 1}(\bar \Omega)}\leq C.
  \end{equation}
  
 Recall that  $v_{\lambda_j}$ is a solution to the following Monge-Amp\`ere equation 
 $$
 (dd^c v_{\lambda_j})^n = g_j, \, \, \text{in} \, \, \Omega, \, \, \, \text{where} \, \, \, g_j :=  (\Vert u_{\lambda_j} \Vert^{-1} - \lambda_j v_{\lambda_j})^n f^n.
 $$
 Observe that the right hand side $g_j $ of this equation  is locally uniformly bounded from below by a positive constant independent of $j$, since $g_j \geq  (- \mu_1 v_{\lambda_j})^n f^n  $ and $v_{\lambda_j} \to \varphi_1 $   uniformly on $\bar \Omega$ and $\varphi_1 <0$ on $\Omega$. 

 Moreover $g_j^{1 \slash n} =  (\Vert u_{\lambda_j}\Vert^{-1} - \lambda_j v_{\lambda_j}) f \in C^{0,1}(\bar \Omega)$ and $\Vert g_j^{1 \slash n} \Vert_{C^{0,1}(\Omega)}$ is uniformly bounded by a constant independent of $j$. 
 
 We can again apply the complex Evans-Krylov local argument  and Schauder theory as in the proof of Theorem \ref{thm:new} (Step 2) to conclude that $\f_1\in PSH(\Omega)\cap C^{\infty}(\Omega) \cap C^{1,\bar 1}(\bar \Omega)$ and the pair $(\mu_1,\f_1)$ is a solution of \eqref{eq1} with $\mu_1\leq \lambda_1$. 
 
 \smallskip
 
  {\bf Step 6 :} We prove that $\mu_1\geq \lambda_1.$
 
 Indeed by Lemma \ref{lem2}, there exists $\bar a \in \mathcal A (\Omega)$ such that 
 $$
 L_{\bar a} \varphi_1 = \Big[\det ((\f_1)_{j\bar k})\Big]^\frac{1}{n}=-\mu_1\f_1f.
 $$
 
  Therefore  $\f_1$ is an eigenfunction of $-L_{\bar a}$ associated to the eigenvalue $\mu_1$ and this implies that $\mu_1\geq \gamma_1(\bar a)\geq \lambda_1.$ Since we already know that $\mu_1 \leq \lambda_1,$ we deduce that 
 $$
 \mu_1=\lambda_1.
 $$
 
 {\bf Step 7  :} The last step is to prove that $(\mu_1,\f_1)$ is unique. Let $(\mu,\f)$ be another solution of \eqref{eq1}. Again by Lemma \ref{lem2},  there exists $\tilde a \in \mathcal{A}(\Omega)$ such that $L_{\tilde a}\f=-\mu \, \f f$, which implies that $\mu \geq \lambda_1.$
 
 We claim that there exists   $t>0$, such that $0\leq t(-\f)\leq -\f_1$ on $\bar \Omega$. Indeed, 
 by Hopf's Lemma (\cite[Lemma 6.4.2, p. 347]{Eva10}) we have $  D_\nu \varphi_1  <  0$ on $\partial \Omega$ ($\nu$ being the inner unit normal vector field on $\partial \Omega$). Hence there exists a constant $c_0 > 0$ such that 
 $$
 \varphi_1 (x) \leq - c_0 d(x,\partial \Omega) \, \, \, \text{on} \, \, \, \Omega.
 $$
 
 On the other hand, let $\rho$ be a smooth defining function for $\Omega$; in particular $\vert \nabla \rho\vert > 0$ on $\partial \Omega$. Since $\varphi = 0$ on $\partial \Omega$, it follows that $\varphi = h \rho$ on a neighborhood of $\partial \Omega$, where $h$ is a smooth function near  $\partial \Omega$. Hence there exists a constant $c_1 >0$ such that $0 \leq -\varphi \leq c_1 (- \rho)$ on $\bar \Omega$.  
 Now it is enough to observe that   $d(x,\partial \Omega)$ is equivalent to $ - \rho (x)$ on $\bar \Omega$. In particular there exist a constant $c'_0 > 0$ such that  $ - \rho (x) \leq c'_0 d(x,\partial \Omega) $ on $\bar \Omega$.
 Summarizing we obtain
 $$
 -\varphi \leq c_1 (- \rho) \leq c_1 c'_0 d(\cdot,\partial \Omega) \leq c_1 c'_0 c_0^{-1} (-\varphi_1),  \, \, \,  \text{on} \, \, \, \Omega,
 $$
 which proves the claim.
 	
 Now let $t_0=\max\{t>0:~ 0\leq t(-\f)\leq -\f_1 ~\textnormal{in}~\bar \Omega \}$. Applying Lemma \ref{lem2} in  
 our case with $(\mu,t_0\f)$ and $(\lambda_1, \f_1)$, we find $a\in \mathcal A(\Omega)$ such that 
 $$	
 L_a(t_0\f-\f_1)+\mu t_0\f f- \lambda_1\f_1f=0.
 $$
 Therefore 
 \begin{align*}
 L_a(t_0\f-\f_1)&= \lambda_1\f_1f - \mu t_0\f f \\&=-\lambda_1f(t_0 \f-\f_1)+(\lambda_1-\mu)ft_0\f.
 \end{align*}
 Since $\lambda_1\leq \mu$ and $\f<0$, we obtain $L_a(t_0\f-\f_1)\geq -\lambda_1f(t_0 \f-\f_1).$ From the  
 fact that $\lambda_1\leq \gamma_1(a)$ and $\phi:=t_0\f-\f_1\geq 0$, we deduce that $L_a(\phi)\geq-\gamma_1(a)\phi f$. Since $\phi$ is bounded and $\phi=0$ in $\partial \Omega$, by Proposition \ref{pro3}, there exists $\theta \in \R $ such that $\phi=\theta    
 \phi_1$ where $\phi_1$ is the positive eigenfunction of the operator $-L_a.$\\
 Thus either $\theta=0$ and we conclude, or $\theta>0$ and using Hopf's Lemma  again (\cite[Lemma 6.4.2, p. 347]{Eva10}), we see that  
 there exists some $\varepsilon>0$ such that $\theta \phi_1 \geq \varepsilon(-\f)$ in $\bar \Omega $, i.e. $t_0\f-\f_1\geq \varepsilon(-\f)$ in $\bar \Omega.$ But, this yields 
	$$
	(-\f_1)\geq (t_0+\varepsilon)(-\f) ~\textnormal{in}~\bar \Omega
	$$
 contradicting the definition of $t_0$. 
 This shows that $\f_1=t_0\f$ and thus $\mu=\lambda_1.$ This proves Theorem \ref{thm1}.
 \end{proof}
	
\smallskip	

Let us  make some interesting remarks which follow from the previous proof.

\begin{rem} 1. The previous proof gives a lower bound for $\lambda_1 = \lambda_1 (f,\Omega).$ Indeed let $u_0 \in PSH (\Omega) \cap C^{\infty} (\bar \Omega)$ be the solution to the equation \eqref{eq319} with $\lambda = 0$ i.e. $(dd^c u_0)^n = f^n \omega^n$ in $\Omega$ and $u_0 = 0$ in $\partial \Omega$. Then by  \eqref{eq:minoration}, we have
$$
\lambda_1 \geq \Vert u_0\Vert^{-1}_{C^0(\bar \Omega)}.
$$
In particular if $\Omega :=  B(a,R) \subset \C^n$ is an euclidean ball of radius $R >0$ then the first eigenvalue 
$\lambda_1 (B(a,R))  = \lambda_1 (B(a,R),1) $ satisfies the following estimate :
\begin{equation} \label{eq:lowerbound}
\lambda_1 (B(a,R)) \geq R^{-2}.
\end{equation}
Indeed the function $u_0 (z) := \vert z-a\vert^2 - R^2$ satisfies the required estimates with $f= 1$ in $B(a,R)$ and 
$\Vert u_0\Vert_{C^0(\bar B(a,R))} = R^2$.

2. Let $F : \Omega' \longrightarrow \Omega$ be a biholmorphic map between bounded strongly pseudoconvex domains
in $\C^n$ which extends smoothly to the boundary.
Then we have 
$$
\lambda_1(\tilde f,\Omega') = \lambda_1(f,\Omega),
$$
where $\tilde f := \vert J_F\vert^{2\slash n} f \circ F $ and  $J_F$ is the Jacobian of $F$ in $\Omega'$. 
Indeed, let $(\lambda_1,\varphi_1)$ the solution to the eigenvalue problem \eqref{eq1} so that 
$(dd^c \varphi_1)^n = (-\lambda_1 \varphi_1)^n f^n \omega^n$ in $\Omega$.
Then 
$$
F^*((dd^c \varphi_1)^n) = (-\lambda_1 \varphi_1\circ F)^n  (f\circ F)^n F^*(\omega^n).
$$
Since $F^*(\omega^n) = \vert J_F\vert^{2} \omega^n$, if we let $\tilde \varphi_1 := \varphi_1\circ F$ and $\tilde f := (f\circ F ) \vert J_F \vert^{2\slash n}$, it follows that
$$
(dd^c \tilde \varphi_1)^n = (- \lambda_1 \tilde \varphi_1)^n \tilde f^n \omega^n,
$$
which proves our claim.
\end{rem}

\subsection{An application}

    We are going to give a further property of the first eigenvalue $\lambda_1 = \lambda_1 (\Omega,1).$ We will consider a smooth function $H \in C^{1,1}(\bar \Omega \times \R)$ satisfying the following condition :
    
    \begin{equation}\label{eq:Quasimonotonie}
    \left\lbrace
    \begin{array}{lcr}
    H(z,t) >  0& \textnormal {for all} &(z,t) \in \bar \Omega \times ]- \infty, 0]\\ \\
   \frac{\partial H}{\partial t} \geq -\lambda_0 & \textnormal{on} & \bar \Omega \times ]- \infty, 0],
    \end{array}
    	\right.
    \end{equation}
    where $\lambda_0 \in  \R$ is a uniform constant. 
    
    Let $\lambda_1$ be the first eigenvalue of the complex Monge-Ampère operator with $f = 1$. Recall that $\lambda_1$ is the supremum of all $\lambda >0$ such that there exists $u_\lambda \in PSH(\Omega) \cap C^2(\bar \Omega)$ satisfying the  equation
    $(dd^c u_\lambda)^n = (1-\lambda u)^n \omega^n$ with $u_\lambda = 0$ in $\partial \Omega$.
    
  We want to solve the following Dirichlet problem.

\begin{equation}\label{eq:GeneralMA}
\left\lbrace
\begin{array}{lcr}
(dd^c u)^n=H^n(z,u) \, \omega^n& \textnormal {in} & \in \Omega \\ \\
u=0 & \textnormal{on} & \partial \Omega,
\end{array}
\right.
\end{equation}
where $u \in PSH(\Omega)$ is the unknown function.

It's well known that  this problem admits a unique  smooth solution if   $\partial_t H \geq 0$ in $\Omega \times ]-\infty , 0]$ by \cite{CKNS85}.

Following Lions \cite{Lions86}, we will prove the following more general result.
\begin{pro}\label{thm:Uniqueness}
 Assume that $H$ satisfies  the assumptions  \eqref{eq:Quasimonotonie} with a constant  $\lambda_0 < \lambda_1$. Then the Dirichlet problem \eqref{eq:GeneralMA} admits a unique solution $u \in PSH(\Omega) \cap C^{\infty} (\bar \Omega).$
\end{pro}
\begin{proof}
 To show the existence part, we just need to build a subsolution and apply Theorem \ref{thm:GuanB}.
	
 Let $C\geq \Vert H(\cdot,0) \Vert_{C^0(\bar \Omega)}>0$. Since $\lambda_0 < \lambda_1$, we can choose $\lambda \in ]\lambda_0, \lambda_1 [$ and set $v := C u_\lambda$, where $u_\lambda$ is a solution to the problem \eqref{eq319}. Then we have in $\Omega$,
	
 \begin{eqnarray*}
  [det (v_{j \bar k})]^{1 \slash n} & =  & C (1 -\lambda u_\lambda) \\
 &\geq  & \Vert H(\cdot,0) \Vert_{C^0(\Omega)} +   \lambda_0(- C u_\lambda).
  \end{eqnarray*}
	On the other hand for $s < 0$ and $z \in \Omega$, we have
	$$
	H (z,0) - H (z,s) = \int_s^0 \partial_t H (z,t) d t \geq - \lambda_0 (- s) = \lambda_0 s.
	$$
	Applying this inequality for $s =Cu_\lambda (z)$ we obtain for $z \in \Omega$, 
	$$
	C  +  \lambda_0(- C u_\lambda) \geq H (z,0)  +  \lambda_0(-C u_\lambda (z)) \geq H (z,C u_\lambda (z)),
	$$
 
 Hence  $C u_\lambda$ is a subsolution to the problem \eqref{eq:GeneralMA}. 
  By Theorem  \ref{thm:GuanB},  it follows that the problem \eqref{eq:GeneralMA} admits a solution $u \in  
 PSH(\Omega) \cap C^{\infty} (\bar \Omega)$.
	
 \smallskip
	
 Now we will prove the uniqueness. 
 If $u,v$ are two solutions to the problem \eqref{eq:GeneralMA}, by Lemma \ref{lem2}, one may find $a \in \mathcal{A}(\Omega)$ such that 
	
 \begin{equation}\label{eq335}
 \left\lbrace
 \begin{array}{lcr}
 L_a(u-v) = H(\cdot,u)-H(\cdot,v)& \textnormal {in} & \Omega \\
 u-v=0 & \textnormal{in} & \partial \Omega\\
 u-v\in C^{2}(\Omega) \cap C^0(\bar \Omega).& &\\
 \end{array}\right.
 \end{equation}
 We have for $z \in \Omega$,
 $$
 H(z,u(z))-H(z,v(z)) := c (z) (u(z) -v(z)),
 $$
 where 
 $$
 c(z) :=\int_{0}^1\frac{\partial H}{\partial s} \left(z, v (z)+s(u(z)-v(z))\right)ds.
 $$ 
 
  Let $w : =u-v$, then \eqref{eq335} becomes
 \begin{equation}\label{eq336}
 \left\lbrace
 \begin{array}{lcr}
 (L_a-c) w = 0& \textnormal {in} & \Omega \\
 w=0 & \textnormal{on} & \partial \Omega\\
 w\in C^{2}(\Omega) \cap C^0(\bar \Omega).& &\\
 \end{array} \right.
	\end{equation}
	
	Let $\phi_1>0$ be an  eigenfunction of $-L_a$ with eigenvalue $\gamma_1 = \gamma_1(a)$, i.e. $L_a \phi_1=-\gamma_1 \phi_1$ and $\phi_{1_{|_{\partial \Omega}}}=0$. 
	
	Since $ \frac{\partial H}{\partial t}(z,t) \geq -\lambda_0,$ it follows that  $c\geq -\lambda_0>-\gamma_1$. Hence  $\phi_1>0$ in $\Omega$ and
	\begin{align*}
		(L_a-c)\phi_1&= L_a\phi_1-c\phi_1\\&=
		-(\gamma_1+ c)\phi_1\\&\leq 0.
	\end{align*}
	
 Therefore the operator $L_a-c$ satisfies the maximum principle (see \cite{BNV94}, property (ii) page 48). Hence the problem \eqref{eq336} has a unique solution, which implies that $w=0$ in $\Omega$ i.e. $u=v$ in $\Omega$.	
\end{proof}
\section{A variational formula for $\lambda_1$}
 For the proof of Theorem \ref{thm:variational}, we need to develop a variational approach to the eigenvalue problem using  Pluripotential Theory. 

\subsection{The  finite energy class } 
We assume here that $\Omega \Subset \C^n$ is hyperconvex i.e. it admits a continuous negative plurisubharmonic exhausion. Let us define an important convex class of singular plurisubharmonic functions in $\Omega$ suitable for the variational approach.  
As we explained in section 2.1,  the complex Monge-Ampère operator $(dd^c \cdot)^n$ is well defined on the class of locally bounded plurisubharmonic functions in $\Omega$. 

 Following Urban Cegrell \cite{Ceg98}, we define the class $\mathcal{E}^0 (\Omega)$ as the set of bounded plurisubharmonic functions $\phi$ on $\Omega$ with boundary values $0$ such that $\int_\Omega (dd^c \phi)^n < + \infty$. Then we define 
$\mathcal{E}^1 (\Omega)$ as the set of plurisubharmonic functions $u$ in $\Omega$ such that there exists a decreasing sequence $(u_j)_{j \in \N}$ in the class $\mathcal{E}^0 (\Omega)$  satisfying $u = \lim_j u_j $ in $\Omega$ and $\sup_j \int_\Omega (-u_j) (dd^c u_j)^n < + \infty$. 
It is easy to see from the definition that  $\mathcal{E}^1 (\Omega)$ is a convex cone in $L^1_{loc}(\Omega)$.  

It turns out that the complex Monge-Amp\`ere operator extends to the class $\mathcal{E}^1 (\Omega)$ and is continuous under monotone limits in $\mathcal{E}^1 (\Omega)$. Moreover if $u \in \mathcal{E}^1 (\Omega)$, then $\int_\Omega (-u) (dd^c u)^n < + \infty$ (see  \cite{Ceg98}). 

\subsection{The Monge-Amp\`ere energy functional}
The Monge-Amp\`ere energy functional is defined on the space $\mathcal{E}^1 (\Omega)$  as follows :  for $\phi \in \mathcal{E}^1 (\Omega)$,
\begin{equation} \label{eq:energyF}
E (\phi) :=  \frac{1}{n+1} \int_\Omega (-\phi) (dd^c \phi)^n.
\end{equation}
 It was proved in  \cite{BBGZ13} in the compact setting that $E$ is lower semi-continuous and is $-E$ is a primitive of the complex Monge-Amp\`ere operator (see also \cite{Lu15}, \cite{ACC12} for domains). More precisely, we have.
 
\begin{lem}   \label{Primitive} 
1) For any smooth path $t \longmapsto \phi_t$ in $\mathcal{E}^1 (\Omega)$ defined in some interval $I \subset \R$, we have
\begin{equation}
\frac{d}{d t} E (\phi_t) =  \int_\Omega (-\dot{\phi}_t) (dd^c \phi_t)^n, \, t \in I.
\end{equation}
In particular if $u, v \in \mathcal{E}^1 (\Omega)$ and $u \leq v$, then  $0 \leq E (v) \leq E (u)$. 

Moreover we have 
\begin{eqnarray*}
\frac{d^2}{d t} E (\phi_t) & = & \int_\Omega (- \ddot{\phi}_t) (dd^c \phi_t)^n   \\
& + &  n  \int_\Omega d \, \dot{\phi}_t \wedge d^c \dot{\phi}_t \wedge (dd^c \phi_t)^{n - 1}, \, t \in I
\end{eqnarray*}
In particular, for any  $u, v \in \mathcal{E}^1 (\Omega)$,  the function $t \longmapsto E( (1-t) u + t v)$ is a convex function in $[0,1]$.

2) The functional $E : \mathcal{E}^1 (\Omega) \longrightarrow \R^+$ is lower semi-continuous on $\mathcal{E}^1 (\Omega)$ for the $L^1_{loc}(\Omega)$-topology.

\end{lem}

The following estimates for the capacity of sublevel sets of functions in $\mathcal{E}^1 (\Omega)$ will be useful (see \cite{CKZ05}). There exists a constant $D_0 > 0$ such that for any $\phi \in \mathcal{E}^1 (\Omega)$ and any $s > 0$, we have

\begin{equation} \label{eq:CapSublevel}
\text{Cap}_\Omega (\{\phi \leq - s \}) \leq \frac{D_0}{s^{n+1}} E (\phi).
\end{equation}

\subsection{A non linear Sobolev-Poincar\'e type inequality}

We will need the following result.
\begin{lem}  \label{lem:Sob} Let $ dV_g := g \omega^n$ be a  positive volume form on $\bar \Omega$ with density $ 0 \leq g \in L^p (\Omega)$, $p > 1$. Then   there exists a constant $A = A(n,\Vert g\Vert_p) > 0$ such that for any $\phi \in \mathcal{E}^1 (\Omega)$, we have
\begin{equation}  \label{eq:Sob}
\int_\Omega (-\phi)^{n+1}  d V_g \leq A \,  E (\phi).
\end{equation}
 Moreover for any sequence $(u_j)$ in $\mathcal{E}^1 (\Omega)$ converging to $u$ in $L^1_{loc}(\Omega)$ such that $M := \sup_j E(u_j) < + \infty$, we have  
 \begin{equation} 
\lim_{j \to + \infty} \int_\Omega (-u_j)^{n+1}  d V_g = \int_\Omega (-u)^{n+1}  d V_g.
\end{equation}
\end{lem}
Observe that for $n=1$ , $E (\phi) = (1\slash 2) \int_\Omega (-\phi) dd^c \phi =  (1\slash 2)  \Vert \nabla \phi \Vert^2_{L^1(\Omega)}$. Hence the inequality \eqref{eq:Sob} in this case is the Poincaré inequality for functions in $\mathcal E^1 (\Omega) \subset W_0^{1,2} (\Omega)$.

\begin{proof} To prove the inequality \eqref{eq:Sob}, we will need the following estimate due to Z. Blocki.
For any $u, v \in \mathcal E^0 (\Omega)$ we have
\begin{equation} \label{eq:Blocki}
\int_\Omega (-u)^{n+1} (dd^c v)^n \leq  \fact{(n+1)}  \, \Vert v\Vert^n_{L^\infty (\Omega)} \int_\Omega (-u) (dd^c u)^n,
\end{equation} 
This can be proved using integration by parts $n$ times  (see \cite{Bl93}).

 By a well know result of Kołodziej \cite{Kol03}, there exists $\phi_0 \in PSH (\Omega) \cap C^0(\bar \Omega)$ such that $\phi_0 = 0$ in $\partial \Omega$ and $(dd^c \phi_0)^n = g d V$ in the sense of currents on $\Omega$.
 
To prove the estimate \eqref{eq:Sob}, it is enough to assume that $\phi \in \mathcal E^0 (\Omega)$.
Then  since we have
 $$
 \int_\Omega (-\phi)^{n+1}  d V_g  = \int_\Omega (-\phi)^{n+1}  (dd^c \phi_0)^n,
 $$
it follows from \eqref{eq:Blocki} that
 $$
 \int_\Omega (-\phi)^{n+1}  d V_g  \leq  \fact{(n+1)}  \,  \Vert \phi_0\Vert^n_{L^\infty (\Omega)} \int_\Omega (-\phi) (dd^c \phi)^n,
 $$
 which proves the required estimate with $A :=  (n+1) \fact{(n+1)}  \, \Vert \phi_0\Vert^n_{L^\infty (\Omega)}$.
 
 Let us prove the second property.
Taking a subsequence if necessary,  we can assume that $u_j \to u$ a.e. in $\Omega$.

  Assume first that $(u_j)$ is uniformly bounded in $\Omega$. Then since the sequence $(-u_j)^{n+1}_{j \in \N}$ converges to $(-u)^{n+1}$ a.e. in $\Omega$ it follows  from the Lebesgue convergence theorem  that  
  \begin{equation} \label{eq:Formula1}
 \lim_{j \to + \infty} \int_\Omega (- u_j)^{n+1}  g d V  = \int_\Omega  (-u)^{n+1} g d V. 
\end{equation}
  
    We now consider the general case.  For  fixed $k, j  \in \N$, 
    where $u^{(k)} := \sup \{u,-k\}$ and $ u_j^{(k)} :=  \sup \{u_j ,  - k\}$.  
    
    Set for $j, k \in \N$, $h_j := (-u_j)^{n+1}$,  $h_j^{(k)} = (-u_j^{(k)})^{n+1}$, $h := (-u)^{n+1}$ and $h^{(k)} :=(- u^{(k)})^{n+1} $. These are Borel  functions in  $L^1(\Omega,dV_g)$ and  we have the following obvious inequalities : 
\begin{eqnarray} \label{eq:FundEq3}
\left\vert \int_\Omega (h_j - h)  d V_g \right\vert &\leq  & \int_\Omega (h_j^{(k)} -h_j ) d V_g + \left\vert \int_\Omega  h_j^{(k)} - h^{(k)}
d V_g\right\vert \\
& + &    \int_\Omega (h^{(k)} - h)  d V_g. \nonumber
\end{eqnarray}

For fixed $k$, the sequence  $(u_j^{(k)})_{j \in \N}$ is a  uniformly bounded sequence of plurisubharmonic functions in $\Omega$.  Then applying the first step, we see that for each $k \in \N$, the second term in (\ref{eq:FundEq3}) converges to $0$ as $j \to + \infty$, while the third term converges to $0$ by the monotone convergence theorem.  It remain to show that the first term converges to $0$ as $k \to + \infty$, uniformly in $j$.
Indeed,  for $j, k \in \N^*$ we have the following obvious estimates 
  \begin{equation} \label{eq:1}
  \int_\Omega\vert h_j - h_j^{(k)} \vert d V_g \leq 2 \int_{\{h_j \geq k^{n+1}\}} h_j d V_g. 
   \end{equation}
   We claim that  the sequence $k \longmapsto \int_{\{h_j \geq k^{n+1}\}} h_j d V_g$ converges to $0$ uniformly in $j$ as $k \to + \infty$.
   Indeed  for fixed $j, k$, we have
  \begin{equation} \label{eq:2}
   \int_{\{h_j \geq k^{n+1}\}} h_j d V_g =  \int_{\{u_j \leq - k\}} (-u_j)^{n+1} g d V
   \end{equation}
   
   On the other hand, given a Borel subset  $B \subset \Omega$, by Kolodziej's theorem \cite{Kol03} there exists $\phi_B \in PSH (\Omega) \cap C^0(\bar \Omega)$ such that $\phi_B =0$ in $\partial \Omega$ and $(dd^c \phi_B)^n = {\bf 1}_B \, g d V$ in the sense of currents on $\Omega$.
    Moreover there exists a uniform contant $C_0 > 0$ such that $ \Vert \phi_B \Vert^n_{L^{\infty} (\Omega)} \leq C_0 \Vert  {\bf 1}_B g \Vert_{L^{p} (\Omega)}$.
   
   Therefore as before,  Blocki's inequality \eqref{eq:Blocki} yields 
   \begin{eqnarray*}
   \int_B (- u_j)^{n+1} g d V = \int_\Omega (- u_j)^{n+1} (dd^c \phi_B)^n 
   & \leq& (n +1)! \, \Vert \phi_B \Vert^n_{L^{\infty} (\Omega)}  \int_\Omega (-u_j) (dd^c u_j)^n   \\
   & \leq & (n+1)! \, C_0 M \Vert  {\bf 1}_B g \Vert_{L^{p} (\Omega)}.
   \end{eqnarray*}
  
  Now since $g^p \in L^{1}(\Omega)$, by absolute continuity, it follows that $\Vert {\bf 1}_B g \Vert^p_{L^{p} (\Omega)} = \int_B g^p d V \to 0$ as $\text{Vol} (B) \to 0$. 
  
  This implies that $\sup_{j \in \N} \int_B (-u_j)^{n+1} g d V \to 0$ as $\text{Vol} (B) \to 0$.
   We want to apply this result to the Borel sets $B :=  \{u_j \leq - k\} $. To estimate their volumes, we first observe that their Monge-Amp\`ere capacity can be controlled using the inequalities (\ref{eq:CapSublevel}) i.e.   for any $j, k \in \N$, we have
  
  $$
  \text{Cap}_\Omega ( \{u_j \leq - k\} ) \leq \frac{D_0}{k^{n+1}} E(u_j) \leq \frac{D_0 M}{k^{n+1}}. 
  $$
  Using  the inequality \eqref{eq:Volcap}, we conclude that for any $k \in \N^*$,
  $$
  \sup_{j \in \N}  \text{Vol} ( \{u_j \leq - k\}) \leq \frac{M D_0 R^{2 n} }{k^{n+1} } \to 0, \, \text{as} \, k \to + \infty.
  $$
  This proves the claim and completes the proof of  the Lemma.
 \end{proof}
 
  \subsection{A Rayleigh quotient type formula}  
 We first use a variational approach to prove the following result of independent interest.
 
 \begin{theorem} \label{thmVariational} Let $\Omega \Subset \C^n$ be a hyperconvex domain and $ dV_g := g \omega^n$ be a  positive volume form on $\bar \Omega$ with density $ 0 \leq g \in L^p (\Omega)$, $p > 1$ such that $\int_\Omega d V_g > 0$. We define the real number
 \begin{equation} \label{eqR}
 \eta_1^n := \inf \left\{\frac{E(\phi)}{I_g (\phi)} ; \phi \in \mathcal E^1(\Omega), w \neq 0 \right\},
 \end{equation} 
 where $I_g(\phi) := \frac{1}{n+1} \int_\Omega (-\phi)^{n +1} d V_g$.

 Then there exists  a function $w \in \mathcal E^1 (\Omega)$  such that 
 \begin{equation} \label{eqR}
 \eta_1^n = \frac{E(w)}{I_g (w)}\cdot
 \end{equation} 
 
  Moreover $(\eta_1,w)$ is a (weak) solution to the eigenvalue problem 
 \begin{equation}\label{eq:WeakEVP}
\left\lbrace
\begin{array}{lcr}
(dd^c w)^n=(-\eta_1 w)^n g \omega^n & \textnormal{on} & \Omega\\
w=0& \textnormal{ in} & \partial \Omega\\
w < 0. & &
\end{array}
\right.
\end{equation}
 
 \end{theorem}
 \begin{proof} By assumption, $\Omega$ admits a continuous negative plurisubharmonic exhaution $\rho$. Then $\rho \in \mathcal E^1(\Omega)$ and any  $w \in \mathcal E^1(\Omega)$ such that $w <  0$  satisfies  $\int_\Omega (-w)^{n +1} d V_g > 0$. Indeed, since $\int_\Omega d V_g > 0$ there exists a compact set $K \Subset \Omega$ such that $ V_g (K) := \int_K d V_g >0$. Therefore   $\int_\Omega (-w)^{n +1} d V_g  \geq  (- \max_K w)^{n +1} V_g (K) > 0$. Hence  $\eta_1$ is a  well defined non negative real number and by homogeneity, we have
\begin{equation} \label{eq:VariatFormula2}
 \eta_1^{n} =   \inf \{ E (w) \,  ; \,  w \in \mathcal E^1 (\Omega), \,   I_g  (w) = 1\}.
 \end{equation} 
 
 Moreover, by  Lemma \ref{lem:Sob}, there exists a constant $A > 0$ such that for any $w \in \mathcal E^1 (\Omega)$
 $$
 \int_\Omega (-w)^{n +1} d V_g \leq A \, E(w).
 $$
 
 In particular  we  conclude that $\eta_1^n \geq A^{-1} >0$.
 
On the other hand, by definition there exists a minimizing sequence $(w_j)_{j \in \N}$ in $\mathcal E^1 (\Omega)$ such that $  I_{g}  (w_j) = 1$ for any $j \in \N$ and
$$
\lim_{j \to + \infty} E (w_j) = \eta_1^n.
$$

By construction it follows that the sequence $(w_j)_{j \in \N}$ is bounded in $L^{n+1} (\Omega,dV_g)$. Extracting a subsequence if necessary we can assume that $(w_j)$ converges weakly to $w \in PSH (\Omega)$ and a.e. in $\Omega$, hence  in $L_{loc}^{1} (\Omega)$.  By semi-continuity of the energy functional, it follows that $w \in \mathcal E^1 (\Omega)$ and $E (w) \leq \lim_{j \to + \infty} E (w_j) = \eta_1^n$.

Since $\sup_j E (w_j) =: C < + \infty$, it follows from  Lemma \ref{lem:Sob} that  
\begin{equation} \label{eq:FundEq1}
\lim_{j \to + \infty} \int_\Omega(-w_j)^{n+1} d V_g =  \int_\Omega(-w)^{n+1} d V_g.
\end{equation}
Hence  $I_{g} (w) = 1$ and 
$w\in \mathcal E^1 (\Omega)$ is an extremal function  for the eigenvalue problem i.e.
 $$
 \eta_1^n = \frac{E (w)}{I_{g} (w)}\cdot 
 $$
 To prove that $(\eta_1,w)$ is a solution to the eigenvalue problem, consider the following functional defined for $\phi \in \mathcal E^1 (\Omega)$, by the formula 
 $$
 F_{g} (\phi) := E (\phi) - \eta_1^n I_{g} (\phi), 
 $$
 and observe that when $\phi$ is smooth then 
 $$
 F_{g}' (\phi) = - (dd^c \phi)^n  +  \eta_1^n (-\phi)^n  dV_g,
 $$
 This means that the eigenvalue equation is the Euler-Lagrange equation of the functional $F_g$ on $\mathcal E^1(\Omega)$.
 
 As observed before, for any $\phi \in \mathcal E^1 (\Omega)$ with $\phi \not \equiv 0$, we have $ I_{g} (\phi) >  0$ and then
 $$
  F_{g} (\phi) :=  I_{g} (\phi) \left(\frac{E (\phi)}{I_{g} (\phi)}  - \eta_1^n\right) \geq 0, 
 $$
 by definition of   $\eta_1. $ Since  $F_g (w) = 0$,
 this means that the functional $F_{g} $ achieves its minimum on  $\mathcal E^1 (\Omega)$ at $w$.
 Therefore $w$ is a kind of "critical point" of the functional $ F_{g}$. To prove this claim, we will use a tricky argument which goes back to \cite{BBGZ13}. 
 
 Fix a "test function" $\psi \in \mathcal{E}^0(\Omega)$ and consider  the path 
 $\phi_t = w + t \psi$ which belongs to $\mathcal E^1 (\Omega)$ when $0 \leq t \leq 1$ by convexity. When $t < 0$,  this is no longer the case, and so we consider its plurisubharmonic envelope $\tilde \phi_t := P (\phi_t)$ i.e. the largest plurisubharmonic function below $\phi_t$ in $\Omega$. Then since $ w \leq \phi_t $ when $t < 0$, it follows that $w \leq P(\phi_t)$ when $t < 0$ hence     $\tilde \phi_t \in \mathcal{E}^1 (\Omega)$ for any $t \in [-1,+1]$ (see \cite{Ceg98}). 
 
 Now consider the one variable function defined for $t \in [-1,+1]$ by 
 $$
 h(t) :=  E \circ P (\phi_t ) - \eta_1^n I_{g} (\phi_t).
 $$
 We claim that the function $h$ is differentiable $ [-1,+1]$, non negative and attains its minimum at $t = 0$. Indeed observe first that $h(0) =0$.
 Moreover since for any $t \in [-1,1]$,  $ \tilde \phi_t  \leq \phi_t < 0$ in $\Omega$, it follows that $  I_{g} (\phi_t) \leq  I_{g} (\tilde \phi_t)$ and then
 $$
 h (t) \geq E ( \tilde \phi_t)  - \eta_1^n I_{g} (\tilde \phi_t) = F_{g} (\tilde \phi_t) \geq 0,
 $$
 for any $t \in [-1,1]$, which proves our claim.
 An important property of the operator $P$ is that for any smooth curve $t \longmapsto \varphi_t$ in $L^1_{loc}(\Omega)$, we have :
 $$
 \frac{d}{d t}(E \circ P) (\varphi_t) =  \int_\Omega (-\dot{\varphi_t}) (dd^c P (\varphi_t))^n,
 $$ 
 for any $t$ (see \cite{BBGZ13}, \cite{Lu15}).
  
 Therefore $h$ is differentiable in $[-1,+1]$ and  by Lemma \ref{Primitive},  we have for any $t \in [-1,1], $
 \begin{eqnarray*}
 h'(t) & = & \frac{d}{d t} E (\tilde \phi_t) - \eta_1^n  \frac{d}{d t}  I_{g} ( \phi_t) \\
 & = & \int_\Omega (- \dot  \phi_t) (dd^c \tilde \phi_t)^n   +  \eta_1^n  \int_\Omega  \dot \phi_t (-\phi_t)^n d V_g. 
 \end{eqnarray*}
 
 Since $h$ achieves its minimum at $0$, it follows that $h'(0) = 0$, which implies the following identity: 
 $$
 \int_\Omega \psi \, (dd^c w)^n  =  \eta_1^n  \int_\Omega  \psi \, (-w)^n d V_g. 
 $$
for any  $\psi \in \mathcal E^0(\Omega) $.  Since any smooth test function $\chi$ in $\Omega$ can be written as $\chi = \psi_1 - \psi_2$, where $\psi_1, \psi_2 \in \mathcal E^0 (\Omega) $, it follows that 
 $$
 (dd^c w)^n  = \eta_1^n  (-w)^n d V_g. 
 $$
in the  sense of currents on $\Omega$.
\end{proof}

\subsection{Proof of Theorem \ref{thm:variational}}

It is an imedaite consequence of the following  result.

 \begin{theorem} \label{thm:Rayleigh} Let $ d \nu_{f} := f^n \omega^n$ be a smooth positive volume form on $\bar \Omega$ and $(\lambda_1,\varphi_1)$ the smooth normalized solution of the eigenvalue problem \eqref{eq1} and let $(\eta_1,w_1)$ be a weak solution to the problem \eqref{eq:WeakEVP} for $g = f^n$. 
 
 Then $\lambda_1 = \eta_1 $ and $\varphi_1 = \theta w_1$ for some constant
 $\theta >0$. In particular
 \begin{equation} \label{eq:Rquotient}
 \lambda_1^n = \frac{\int_\Omega (-\varphi_1) (dd^c \varphi_1)^n}{\int_\Omega (-\varphi_1)^{n +1} d \nu_{f}} = \inf \left\{\frac{\int_\Omega (-\phi) (dd^c \phi)^n}{\int_\Omega (-\phi)^{n +1} d \nu_{f}} ; \phi \in \mathcal E^1(\Omega), \phi \neq 0 \right\}\cdot
 \end{equation} 
 \end{theorem}
 
 \begin{proof} Since $(\lambda_1,\varphi_1)$ is the solution of  the eigenvalue problem \eqref{eq1} given by Theorem \ref{thm1}, it follows that 
 $$
 \int_\Omega (-\varphi_1) (dd^c \varphi_1)^n = \lambda_1^n \int_\Omega (-\varphi_1)^{n+1} f^n \omega^n.
 $$
 By the formula \eqref{eqR} we see that $\lambda_1 \geq \eta_1$. 
 
 We first prove that $\eta_1 = \lambda_1$. We argue by contradiction. Assume that $\eta_1 < \lambda_1=\mu_1$. From  step 3 of the proof of Theorem \ref{thm1}, it follows that there exists $u_{\eta_1} \in PSH (\Omega) \cap C^{2}(\bar \Omega)$ such that
 $$
 (dd^c u_{\eta_1})^n = (1- \eta_1 u_{\eta_1})^n f^n \omega^n, \, \, u_{\eta_1} = 0 \, \, \text{in} \, \, \partial \Omega.
 $$
  Moreover $w_1$ is a weak (super)-solution to the problem \eqref{eq:WeakEVP}.  We want to compare $u_{\eta_1}$ and $w_1$ by applying the comparison principle. 
    
  We first claim that  $w_1$ is bounded in $\Omega$ and $w_1 = 0$ in $\partial \Omega$. Indeed by Lemma \ref{lem:Sob}, we have $ \mathcal E^1(\Omega) \subset L^{n+1} (\Omega)$, hence $w_1 \in L^{n+1} (\Omega)$. 
  Set $p := \frac{n+1}{n} > 1$. Since $f$ is bounded in $\Omega$, it follows that
  $$
  \int_\Omega [(-\eta_1 w_1)^n f^n]^p \omega^n \leq\int_\Omega (-\eta_1 w_1)^{n+1} f^{n+1} \leq \Vert \eta_1 f\Vert_{L^{\infty}(\Omega)}^{n+1} \int_\Omega (-w_1)^{n+1} \omega^n < \infty.
  $$
  This means that the density $g:= (-\eta_1 w_1)^n f^n \in L^p(\Omega)$ with $p > 1$. By  a theorem of Kolodziej \cite[Theorem 3]{Kol96}, there exists a unique function $v \in PSH (\Omega) \cap C^0(\bar \Omega)$ such that $(dd^c v)^n = (-\eta_1 w_1)^n f^n \omega^n$ on $\Omega$ and $v = 0$ in $\partial \Omega$. 
  Moreover we also have $ v \in \mathcal E^1 (\Omega)$. This means that we have two weak solutions $v, w_1 \in \mathcal E^1(\Omega)$ of the complex Monge-Ampère equation 
  $(dd^c \phi)^n = (-\eta_1 w_1)^n f^n \omega^n$.  By the uniqueness theorem of Cegrell \cite[Theorem 6.2]{Ceg98}, it follows that $w_1 = v  \in C^0(\bar \Omega)$ and $w_1 = 0$ in $\partial \Omega$. This proves our claim.
  
Now we have  $(dd^c w_1)^n = (-\eta_1 w_1)^n f^n \omega^n \leq t^n f^n \omega^n$ on $\Omega$, where $t := \eta_1 M$ and $M := \max_{\bar \Omega} w_1$. Moreover we also  have 
$ (dd^c u_{\eta_1})^n \geq f^n \omega^n$ on $\Omega$ and then  $(dd^c w_1)^n  \leq (dd^c t u_{\eta_1})^n$  on $\Omega$. Since $t u_{\eta_1} = 0 = w_1$ in $\partial \Omega$, it follows from  the comparison principle that $\underline u :=  t u_{\eta_1} \leq  w_1$ in $\Omega$. 
 
  Observe that  
 \begin{eqnarray*} 
 (dd^c \underline u)^n &= & t^n (1- \eta_1 u_{\eta_1})^n  f^n \omega^n \\
 &\geq & t^n \left(1 + (-\eta_1 u_{\eta_1})^n\right) f^n \omega^n \\
 & = & ((- \eta_1 \underline u)^n + t^n) f^n \omega^n.
\end{eqnarray*}
  Hence $ \underline u$ is a smooth strict subsolution to the problem \eqref{eq:WeakEVP} and $w_1$ is a continuous supersolution to the problem \eqref{eq:WeakEVP} such that $\underline u \leq w_1$ in $\Omega$. 
   It follows from Theorem \ref{thm:new} that there exists a solution $\varphi \in PSH (\Omega) \cap C^{1,\bar1} (\bar \Omega)$ to the problem \eqref{eq:WeakEVP} such that $\underline u  \leq \varphi \leq w_1$ in $\Omega$. Therefore if $\theta := \Vert \varphi\Vert_{C^0(\bar \Omega)}$, then $(\eta_1,\theta \varphi)$ is another smooth solution to the eigenvalue problem \eqref{eq1}.  By the uniqueness property in Theorem \ref{thm1}, it follows that $\eta_1 = \lambda_1$.
 This contradiction leads to the conclusion  $\eta_1 = \lambda_1$ . Now the fact that  $w_1 = \theta \varphi_1$ for some positive constant $\theta >0$ follows from the uniqueness property in Proposition \ref{pro:unicitefaible} below.
 \end{proof}

\begin{prop}\label{pro:unicitefaible}  Let $(\lambda_1, \varphi_1)$ be the normalized smooth solution to the eigenvalue problem  \eqref{eq1}. 
Assume that $u\in \mathcal E^1 (\Omega)$ is a weak solution to the following problem
  \begin{equation}\label{eq2}
 \left\lbrace
  \begin{array}{lcr}
 (dd^c u)^n=(-\lambda_1 u)^n f^n\omega^n & \textnormal{on} & \Omega\\
 u=0& \textnormal{ in} & \partial \Omega.\\
 \end{array}
 \right.
 \end{equation}
 Then there exists a positive constant $\theta>0$ such that  $u=\theta \f_1$.
\end{prop}

The main idea of the proof of this result is due to  Chinh H. Lu.

 \begin{proof} 
 First observe that $u \in PSH (\Omega) \cap  C^{\alpha} (\bar\Omega)$ for some $\alpha \in ]0,1[$. Indeed as in the last proof, we see  that the density $g := (-\lambda_1 u)^n f^n \in L^{1+1\slash n} (\Omega)$. Therefore by \cite{Ch15}  (see also \cite{GKZ08}), there exists $ \phi \in PSH (\Omega) \cap  C^{\alpha} (\bar\Omega)$ such that $(dd^c \phi)^n = g \omega^n$ in $\Omega$ and $\phi = 0$ in $\partial \Omega$.
 By the  uniqueness theorem of Cegrell for solutions in $\mathcal E^1(\Omega)$ we conclude that $u = \phi \in PSH (\Omega) \cap  C^{\alpha} (\bar\Omega)$ (see \cite[Theorem 6.2]{Ceg98}).
 
 By the mixed Monge-Amp\`ere inequalities (\cite{Kol03}, \cite{Din09}), we have
 $$
 dd^cu  \wedge  (dd^c \f_1)^{n-1} \geq (-\lambda_1 u f)(- \lambda_1 \f_1 f )^{n-1}\omega^n = (- \lambda_1 u) G \omega^n,
  $$
 weakly on $\Omega$, where  $G := (-\lambda_1 \varphi_1)^{n-1} f^n$ is a smooth function in $\bar \Omega$.
 	
 Let us consider the following linear elliptic second order operator
 $$
  L \phi := dd^c \phi \wedge (dd^c\f_1)^{n-1} \slash \omega^n.	
$$
 Then $L \varphi_1 = (- \lambda_1 \varphi_1) G$, and
 
 \begin{equation}\label{eq4}
  Lu \geq -\lambda_1 u G,
 \end{equation}
 weakly on $\Omega$.
 
  Since  $-\lambda_1 u G \in C^{\alpha} (\bar \Omega)$, it follows from Schauder theory that  the following Dirichlet problem
 	
  \begin{equation}\label{eq5}
  \left\lbrace
 \begin{array}{lcr}
 L v =-\lambda_1 u G & \textnormal{on} & \Omega\\
 v=0& \textnormal{ in} & \partial \Omega,\\
  \end{array}
 		\right.
 	\end{equation}
 admits a solution  $v \in C^{2,\alpha}(\Omega)$. 
 
 Indeed the  operator $L$ is elliptic on $\Omega$ and then it satisfies the maximum principle \cite[Theorem 3.5]{GT98}. Moreover the  operator $L$ is locally uniformly elliptic and its coefficients are locally H\"older continuous on $\Omega$.
 Therefore we can apply the observation of  \cite[Section 6.6]{GT98}. To be more precise we can apply the Perron method to solve the Dirichlet problem \eqref{eq5} as explained after the statement of \cite[Theorem 6.11]{GT98}. Indeed observe that $u$ is a subsolution to this problem, hence  the upper enveloppe $v$ of all subsolutions to the Dirichlet problem \eqref{eq5} is well defined and is a bounded subsolution such that $u \leq v \leq 0$ in $\Omega$. To show that $v \in C^{2,\alpha}(\Omega)$ is a solution to  the Dirichlet problem \eqref{eq5}, one can use the classical balayage process  using the same arguments as in the proof of \cite[Theorem 2.12]{GT98}, thanks to the compactness result provided by the interior Schauder estimates of  \cite[Corollary 6.3]{GT98}.
  
 By \eqref{eq4}, we have $L u  \geq -\lambda_1u G = L v$, hence  $u\leq v\leq 0$ in $\bar \Omega$ and then $L v \geq -\lambda_1 v G$
 on $\Omega$.
	
 Now applying Proposition \ref{pro3}, we deduce that there exists $\theta > 0$ such that
 $v =\theta \f_1.$ Hence   $L v =-\lambda_1v G$ in $\Omega$ and $v =0$ in $\partial \Omega$.
 	By \eqref{eq5}, we deduce that $u=v$, hence $u=\theta \f_1.$
 \end{proof}

\subsection{The monotonicity property of $\lambda_1 (\Omega)$}
Let $\Omega' \Subset \Omega \Subset \C^n$ be two bounded strongly pseudoconvex domains and $0 < f \in C^{\infty}(\bar \Omega)$. Then we have the following comparison theorem
\begin{theorem} Let $f' := f_{\Omega'}$ be the restriction of $f$ to $\Omega'$. Then we have 
$$
\lambda_1 (\Omega,f) \leq \lambda_1 (\Omega',f').
$$
\end{theorem}
\begin{proof}
We denote by $E_{\Omega} (u) := \frac{1}{n+1} \int_\Omega (-u) (dd^c u)^n$ for $u \in \mathcal E^1(\Omega)$ and 
$I_{\Omega,f} :=  \frac{1}{n+1} \int_\Omega (- u)^{n +1} f^n d V$.

By Theorem \ref{thm:Rayleigh} there exists $w' \in \mathcal E^1(\Omega')$ such that
$$
\lambda_1 (\Omega',f_{\Omega'}) = \frac{E_{\Omega'}(w')}{I_{\Omega',f'} (w')}.
$$
By the subextension Theorem from \cite{CKZ11}, it follows that there exists $w \in \mathcal E^1(\Omega)$ such that $w \leq w'$ in $\Omega'$ and $ E_{\Omega}(w) \leq E_{\Omega'}(w')$. Since $w \leq w'$ in $\Omega'$, it follows that $I_{\Omega',f'} (w') \leq  I_{\Omega,f} (w)$, hence
$$
 \frac{E_{\Omega'}(w')}{I_{\Omega',f'} (w')} \geq  \frac{E_{\Omega}(w)}{I_{\Omega,f} (w)}.
 $$
 By Theorem \ref{thm:Rayleigh} we conclude that $ \lambda_1 (\Omega',f') \geq \lambda_1 (\Omega,f).$
\end{proof}

From \eqref{eq:lowerbound} we deduce the following lower bound for the first eigenvalue.
\begin{coro} Let $  \Omega \Subset \C^n$ be a bounded strongly pseudoconvex domain and $R := {diam}(\Omega) \slash 2$. Then the eigenvalue  $\lambda_1(\Omega) = \lambda_1(\Omega,1)$ satisfies the following estimate :
$$
\lambda_1(\Omega)  \geq R^{-2}.
$$
\end{coro}

\smallskip



\smallskip

\begin{thebibliography}{widestlabel}

\bibitem[1]{ACC12} P. Åhag,  U. Cegrell, R.  Czyz :  {\it On Dirichlet's principle and problem.} Math. Scand. 110 (2), 235-250 (2012).
	 
\bibitem[2]{ACKPZ09}  P. Åhag,  U. Cegrell, S. Kolodziej, H.H. Pham, A. Zeriahi: Partial Energy and Integrability Exponents, Adv. Math. 222 (2009), no. 6, 2036–2058.
  
  \bibitem[3]{BZ23} P. Badiane and  A. Zeriahi : {\it The Eigenvalue Problem for the Complex Monge-Ampère Operator.} J. Geom. Anal. 33 (2023), no. 12, Paper No. 367, 44 pp.
  
  \bibitem[3]{BZ26} P. Badiane and  A. Zeriahi : {\it The Eigenvalue Problem for the Complex Monge-Ampère Operator.} Preprint arXiv .
  
  \bibitem[4]{BLZ26} P. Badiane, Chinh H. Lu and  A. Zeriahi : {\it Erratum : The Eigenvalue Problem for the Complex Monge-Ampère Operator.} (Submitted to JGEA 2026). 
   
   \bibitem[5]{BT76} E. Bedford and B. A. Taylor : {\it The Dirichlet problem for a complex Monge-Amp\`ere equation.} Invent. Math., 37 (1), 1-44 (1976).	
	
 \bibitem[6]{BT82} E. Bedford and B. A. Taylor :	{\it A new capacity for plurisubharmonic functions.} Acta Math. 149 , no. 1 (2), 1-40 (1982).
 
 \bibitem[7]{Bl93} Z. Blocki : {\it  Estimates for the complex Monge-Amp\`ere operator.}  Bulletin of the Polish Academy of Sciences 41, 151-157 (1993).

\bibitem[8]{Bl09} Z. B{\l}ocki, A gradient estimate in the Calabi-Yau theorem, Math. Ann. 344 (2009), no. 2, 317--327.


  \bibitem[9]{BNV94} H. Berestycki, L. Nirenberg, S.R.S. Varadhan : {\it   The first eigenvalue and maximum principle for second order elliptic differential operators in general domains.} Comm. Pure Appl. Math. 47 (1), 47-92 (1994).


 \bibitem[10]{BBGZ13} R. Berman, S. Boucksom, V. Guedj, A. Zeriahi : {\it A variational approach to complex Monge-Amp\`ere equations.} Publ. Math. Inst. Hautes Études Sci. 117, 179-245  (2013).
	
 \bibitem[11]{CKNS85} L. Caffarelli, J. J. Kohn, L. Nirenberg, J. Spruck : {\it The Dirichlet problem for nonlinear second-order elliptic equations. II. Complex Monge-Amp\`ere, and uniformly elliptic, equations.} Comm. Pure Appl. Math. 38 (2), 209-252 (1985).
	
\bibitem[12]{CKZ11}  U. Cegrell, S. Kolodziej,  A. Zeriahi : {\it  Maximal subextensions of plurisubharmonic functions. }Ann. Fac. Sci. Toulouse Math. 20 (6), 101–122  (2011).
	
 \bibitem[13]{CNS84} L. Caffarelli, L. Nirenberg, J. Spruck :  {\it The Dirichlet problem for nonlinear
	second-order elliptic equations. I. Monge-Amp\ `ere. } Comm. Pure Appl. Math. 37, 369-402  (1984).	
	
 \bibitem[14]{Ceg98}  U. Cegrell : {\it Pluricomplex energy.} Acta Math. 180 (2), 187-217  (1998). 
	
 \bibitem[15]{Ch15} M. Charabati : {\it Hölder regularity for solutions to complex Monge-Amp\`ere equations.}
Ann. Polon. Math. 113 (2), 109-127  (2015).

 \bibitem[16]{CKZ05}  U. Cegrell, S. Kolodziej,  A. Zeriahi : {\it Subextension of plurisubharmonic functions with weak singularities}. Math.  Zeit.  250, 7-22 (2005).
	
 \bibitem[17]{Din09} S. Dinew : {\it An inequality for mixed Monge-Amp\`ere measures}. 
 Math. Z.  262, 1-15 (2009).
 
\bibitem[18]{DZZ11} S. Dinew, X. Zhang, X.W. Zhang : {\it  The $C^{2,\alpha}$ estimate of complex Monge-Amp\`ere equation}. 
IUMJ, 60-5 (2011), 1713-1722.

 \bibitem[19]{Eva10} L. C. Evans : {\it Partial Differential Equations.} Graduate Studies in Mathematics 19, Second edition, American Mathematical Society (2010).
	
\bibitem[20]{Gav77} B. Gaveau : {\it  M\'ethode de contr\^ole optimal en analyse complexe, I.} J. Funct. Anal. 25,  391-411 (1977).

\bibitem[21]{GT98} D. Gilbarg and N. Trudinger : {\it Elliptic Partial Differential Equations of Second Order (2nd edn.)}. Grundlehr. der Math. Wiss.,  Springer, Berlin (1998).
	
	
\bibitem[22]{Guan98} B. Guan : {\it The Dirichlet problem for complex Monge-Amp\`ere equations and Regularity of the Pluricomplex  Green Function.} Comm. Anal. Geom., 6 (4), 687–703  (1998).
	
	
\bibitem[23]{GL10} B. Guan and Q. Li : {\it Complex Monge-Amp\`ere equations and totally real submanifolds.} Adv. Math. 225, 1185-1223  (2010).
	
 \bibitem[24]{GKZ08} V. Guedj, S. Koldziej and A. Zeriahi : {\it Hölder continuous solutions to Monge-Amp\`ere equations.} Bull. Lond. Math. Soc. 40 (6), 1070–1080  (2008). 
	
 \bibitem[25]{GZ17} V. Guedj and A. Zeriahi : {\it Degenerate Complexe Monge-Amp\`ere Equations.} EMS Tracts in Mathematics 26 (2017).

	 
\bibitem[26]{Kol96}  S. Kolodziej : {\it  Some sufficient conditions for solvability of the Dirichlet problem for the complex Monge-Amp\`ere operator}.  Ann. Pol. Math. 65 (1), p. 11-21  (1996).
	
 \bibitem[27]{Kol03}  S. Kolodziej : {\it Equicontinuity of families of plurisubharmonic functions with bounds on their Monge-Amp\`ere masses.} Math. Z. 240 (4), 835-847  (2002).
	
 \bibitem[28]{KR89} N.D. Koutev, I. P. Ramadanov : {\it Valeurs propres radiales de l'opérateur de Monge-Amp\`ere complexe.} Bull. Sc. Math. 113 (2),  195-212  (1989).
	
 \bibitem[29]{KR90} N.D. Koutev, I. P. Ramadanov : {\it An eigenvalue problem for the complex Monge-Amp\`ere operator in pseudoconvex domains.} Ann. Inst. Henri Poincaré. 7 (5), 493-503 (1990).
	
	
 \bibitem[30]{Lions86} P. L. Lions : {\it Two remarks on the Monge-Amp\`ere equations.} Ann. Mat. Pura Appl. 142 (4), 263-275 (1986).
	
 \bibitem[31]{Lu15} C. H. Lu : {\it  A variational approach to complex Hessian equations in $\C^n$.} J. Math. Anal. Appl. 431 (1), 228-259  (2015).	
	
 \bibitem[32]{NP92} R. Nussbaum, Y. Pinchover : {\it On variational principles for the generalized principal eigenvalue of second order elliptic operators and applications.}  Journal d'Analyse Mathématique. 59, 161-177 (1992).
	
	
 \bibitem[33]{Tso90} K. Tso : {\it On a real Monge-Amp\`ere functional.} Invent. Math. 101, 425-448 (1990).
	
 \bibitem[34]{Wang12} Y. Wang : {\it On the $C^{2,\alpha}$ regularity of the complex Monge–Amp\`ere equation.} Math. Research Letter 19, 939–946 (2012). 	
	
		
\end{thebibliography}
\end{document}